\documentclass[reqno]{amsart}

\usepackage{geometry}
\usepackage{dsfont}
\usepackage{MnSymbol}
\usepackage{color}
\usepackage{todonotes}
\usepackage{verbatim}

\usepackage{graphicx}
\usepackage{subfigure}
\usepackage[all]{xy}
\usepackage{tikz,pgfplots}
\usetikzlibrary{calc}
\usetikzlibrary{arrows,petri, topaths}
\usepackage{tkz-berge}
\usepackage{caption}
\usepackage[foot]{amsaddr}
\usepackage{listings}

\definecolor{colKeys}{rgb}{0,0,1} 
\definecolor{colIdentifier}{rgb}{0,0,0} 
\definecolor{colComments}{rgb}{0,1,0.3} 
\definecolor{colString}{rgb}{0,0.5,0} 

\definecolor{dkgreen}{rgb}{0,0.6,0} 
\definecolor{gray}{rgb}{0.5,0.5,0.5} 
\definecolor{lightgray}{rgb}{0.9,0.9,0.9} 

\title[A variational scheme for Fokker-Planck equations]
{A fully discrete variational scheme for solving nonlinear Fokker-Planck equations in higher space dimensions}

\author{Oliver Junge and Daniel Matthes and Horst Osberger}
\address{Zentrum Mathematik \\ TU M\"unchen \\ Boltzmannstr. 3 \\ D-85748 Garching \\ Germany}
\email{oj@tum.de, matthes@ma.tum.de, osberger@ma.tum.de}

\thanks{This research was supported by the DFG Collaborative Research Center, TRR 109, ``Discretization in Geometry and Dynamics.''}

\newcommand{\setR}{\mathbb{R}}
\newcommand{\setRp}{\mathbb{R}_{\geq0}}
\newcommand{\setRpp}{\mathbb{R}_{>0}}
\newcommand{\setN}{\mathbb{N}}

\newcommand{\dd}{\,\mathrm{d}}
\newcommand{\dn}{\mathrm{d}}
\newcommand{\id}{\operatorname{id}}
\newcommand{\eins}{\mathds{1}}
\newcommand{\nrml}{\mathbf{n}}
\newcommand{\ball}{\mathbb{B}}
\newcommand{\bigO}{\mathcal{O}}
\newcommand{\rf}{\text{ref}}
\newcommand{\sphere}{\mathbb{S}}

\newcommand{\dff}{\mathrm{D}}

\renewcommand{\div}{\operatorname{div}}
\renewcommand{\det}{\operatorname{det}}

\newcommand{\argmin}{\operatorname*{arg\,min}}
\newcommand{\tr}{\operatorname{tr}}

\newcommand{\proj}{\Pi}
\newcommand{\isom}[2]{\mathbf{J}_{#1}\left[#2\right]}

\newcommand{\wass}{\mathcal{W}_2}
\newcommand{\tm}{{\mathfrak{T}}}
\newcommand{\T}{{\mathfrak{t}}}
\newcommand{\dens}{\mathcal{P}(\Omega)}
\newcommand{\uelo}{\mathbf{v}}
\newcommand{\velo}{\mathbf{v}}
\newcommand{\rfd}{\bar u}

\newcommand{\prss}{\mathrm{P}}
\newcommand{\anrj}{\mathbb{E}}
\newcommand{\nrjx}{\mathbf{E}}

\newcommand{\xspc}{\mathfrak{X}}
\newcommand{\txspc}{\mathrm{T}\mathfrak{X}}
\newcommand{\maps}{\mathrm{Diff}^+}
\newcommand{\theh}{\delta}
\newcommand{\basis}{\mathfrak{B}}
\newcommand{\bas}{\mathfrak{b}}
\newcommand{\basQ}{\mathfrak{q}}

\newcommand{\intom}{\int_{\Omega}}
\newcommand{\I}{\mathcal{I}}
\newcommand{\ki}{\mathbf{k}}
\newcommand{\li}{\mathbf{\ell}}

\newtheorem{thm}{Theorem}
\newtheorem{prp}[thm]{Proposition}
\newtheorem{lem}[thm]{Lemma}
\newtheorem{cor}[thm]{Corollary}
\newtheorem{rmk}[thm]{Remark}

\providecommand{\gridplot}[1]{
	\begin{tikzpicture}
	\tikzstyle{every node}=[font=\footnotesize]
    \begin{axis}[enlargelimits=false,
    	axis on top,
		width=0.29\textwidth,
		height=0.29\textwidth,
		xlabel={$x$},xlabel style={at={(0.5,0)}},
		ylabel={$y$},ylabel style={at={(0,0.5)}},
		xtick={0,1},ytick={0,1},
		xticklabels={0,1},yticklabels={\empty,1}]
        \addplot graphics [xmin=0,xmax=1,ymin=0,ymax=1]{#1};
    \end{axis}
	\end{tikzpicture}
}

\providecommand{\densplot}[1]{
	\begin{tikzpicture}
	\tikzstyle{every node}=[font=\footnotesize]
	\begin{axis}[enlargelimits=false,
		grid=both,
		width=0.29\textwidth,
		height=0.29\textwidth,
		minor tick num=1,
		xlabel={$y$},xlabel style={at={(0.1,0.1)}},
		ylabel={$x$},ylabel style={at={(0.75,0.05)}},
		xtick={0,1}, ytick={0,1}, ztick={0,1,2,3,4,5},
		xticklabels={0,1},yticklabels={\empty,1},
		zmin=0, zmax=5,
		xmin=0, xmax=1,
		ymin=0, ymax=1
		]
	\addplot3 graphics[
		points={
			(0,1,0.33) => (21.7,107.9-73) 
			(0.53,0.27,2.1) => (84.5,107.9-54)
			(1,0,0.81) => (125.9,107.9-73.6)
			(0,0,0.33) => (59.8,107.9-91.9)
		},zmin=0] {#1};
	\end{axis}
	\end{tikzpicture}
}

\begin{document}

\begin{abstract}
  We introduce a novel spatio-temporal discretization for nonlinear Fokker-Planck equations 
  on the multi-dimensional unit cube.
  This discretization is based on two structural properties of these equations:
  the first is the representation as a gradient flow of an entropy functional in the $L^2$-Wasserstein metric,
  the second is the Lagrangian nature, 
  meaning that solutions can be written as the push forward transformation of the initial density under suitable flow maps.
  The resulting numerical scheme is entropy diminishing and mass conserving. 
  Further, the scheme is weakly stable, which allows us to prove convergence under certain regularity assumptions.
  Finally, we present results from numerical experiments in space dimension $d=2$.
\end{abstract}

\maketitle

\section{Introduction}\label{sec:intro}
%
In this paper, we propose and study spatio-temporal discrete approximations for solutions 
to the following initial boundary value problem for a non-linear Fokker-Planck equation:
\begin{align}
  \label{eq:eq}
  \partial_t u = \Delta\prss(u) + \div(u\nabla V) &\quad\text{for $(t;x)\in\setRpp\times\Omega$},\\ 
  \label{eq:bc}
  \nrml\cdot\nabla\prss(u) = 0 &\quad\text{for $(t;x)\in\setRpp\times\partial\Omega$},\\ 
  \label{eq:ic}
  u(0;x)=u^0(x)\ge0 &\quad\text{for $x\in\Omega$}. 
\end{align}
The domain $\Omega=(0,1)^d$ is the $d$-dimensional unit cube,  
$\prss:\setRp\to\setR$ is the pressure function,
and $V:\Omega\to\setR$ is a smooth external potential.
We assume that $\prss$ is continuous, is smooth on $\setRpp$, and is strictly increasing.
Typical choices for the pressure function are $\prss(u)=u$ (heat equation),
and $\prss(u)=u^m$ with some exponent $m>1$ (porous medium equation) or $0<m<1$ (fast diffusion equation).
To keep technicalities to a minimum, we shall impose some further restrictions on $\prss$, $V$ and $u^0$ later,
see Section \ref{sct:techhypos}.

\subsection{Concept of the discretization}
The discretization approach presented here has two main features: 
first, it is variational, and second, it is fully Lagrangian.

Let us start with a discussion of the \textbf{Lagrangian structure} of the evolution:
consider the Fokker-Planck equation \eqref{eq:eq} 
as a transport equation for the density $u$,
\begin{align}
  \label{eq:trapo}
  \partial_tu = -\div\big(u\,\uelo[u]\big),
\end{align}
with a vector field $\uelo[u]:\Omega\to\setR^d$ that depends sensitively on $u$,
namely
\begin{align}
  \label{eq:velo}
  \uelo[u] = -\nabla\big[h'(u) + V\big],
\end{align}
where $h:\setRp\to\setR$ is implicitly defined (up to an affine normalization) by $\prss'(r)=rh''(r)$ for all $r>0$. 
In view of the boundary conditions \eqref{eq:bc}, 
this representation implies conservation of non-negativity and total mass $M$ of the solutions $t\mapsto u_t=u(t;\;\cdot\;)$, two key properties that are inherited by our discretization.

Moreover, the transport formulation \eqref{eq:trapo} implies 
that the solution $u_t$ at any time $t\ge0$ can be written 
as the push-forward of an a priori fixed reference density $\rfd$ with total mass $M$ on $\Omega$ 
under a suitable diffeomorphism $\tm_t:\Omega\to\Omega$,
i.e.,
\begin{align}
  \label{eq:pushit}
  u_t = (\tm_t)_\# \rfd = \frac{\rfd}{\det\dff\tm_t}\circ(\tm_t)^{-1}.
\end{align}
Canonical choices for $\rfd$ are either the unit density on $\Omega$, or the initial condition $u^0$.
The $\tm_t$ are far from unique, 
but there is a distinguished family $(\tm_t)_{t\ge0}$,
namely the one obtained as flow maps for the ordinary differential equation $\dot x=\uelo[u_t](x)$.
In other words, one chooses the maps $\tm_t$ such that 
\begin{align}
  \label{eq:tmfromu}
  \partial_t\tm_t=\uelo[u_t]\circ\tm_t,  
\end{align}
which determines the $\tm_t$ uniquely up to a global change of variables that is given by the initial map $\tm_0$.
The latter is arbitrary but subject to the consistency condition $(\tm_0)_\#\rfd=u^0$.

For the Lagrangian discretization, we approximate the evolution of the maps $\tm_t$ in a numerically tractable form.
Our basic idea here is \emph{not} to restrict the diffeomorphisms $\tm_t$ themselves to an ansatz space,
but their time increments. 
More specifically, given a time step width $\tau>0$, 
our ansatz for the approximation $\tm_\Delta^n$ of the map $\tm_t$ at time $t=n\tau$
is the concatenation of $n$ close-to-identity diffeomorphisms $\T_\Delta^1,\ldots,\T_\Delta^n$:
\begin{align}
  \label{eq:concat}
  \tm_\Delta^n = \T_\Delta^n\circ\T_\Delta^{n-1}\circ\cdots\circ\T_\Delta^1\circ \tm_\Delta^0.
\end{align}
The $\T_\Delta^m$ in \eqref{eq:concat} are inductively determined inside a fixed finite-dimensional manifold of diffeomorphisms.
Therefore, each $\T_\Delta^m$ has a simple form, but the complexity of $\tm_\Delta^n$ increases rapidly with $n$.

For our definition of the recursion relation of the incremental diffeomorphisms $\T_\Delta^m$,
we build on the \textbf{variational structure} of \eqref{eq:eq}. 
Observe that the evolution \eqref{eq:tmfromu} for the diffeomorphisms $\tm_t$ is actually autonomous,
since one can substitute $u_t=(\tm_t)_\#\rfd$:
\begin{align}
  \label{eq:evans}
  \partial_t\tm_t = \uelo\big[(\tm_t)_\#\rfd\big]\circ\tm_t.
\end{align}
Now, it turns out that \eqref{eq:evans} is a \emph{gradient flow} 
on the manifold of diffeomorphisms on $\Omega$ with respect to the $L^2(\rfd)$-metric;
the respective potential $\nrjx$ is given in \eqref{eq:newpot} below.
This structure is not a coincidence, but is a consequence of the fact that the original equation \eqref{eq:eq}
is a gradient flow of an appropriate entropy functional $\anrj$ with respect to the $L^2$-Wasserstein metric;
these relations are detailed in Section \ref{sct:gf} below.

The variational discretization of this gradient flow is now performed using minimizing movements,
which are a variational formulation of the implicit Euler scheme.
Given the diffeomorphism $\tm_\Delta^{n-1}$ from the previous step,
the next diffeomorphism $\tm_\Delta^n=\T_\Delta^n\circ\tm_\Delta^{n-1}$ is obtained by choosing the increment $\T_\Delta^n$
as the minimizer of the functional
\begin{align}
  \label{eq:nmm}
  \T\mapsto\frac1{2\tau}\big\|\T\circ\tm_\Delta^{n-1}-\tm_\Delta^{n-1}\big\|_{L^2(\rfd)}^2
  + \nrjx\big(\T\circ\tm_\Delta^{n-1}\big)
\end{align}
over our finite-dimensional set of diffeomorphisms.
Thanks to its variational nature,
this particular discretization approach has --- at least --- two valuable analytical properties.
The first, obvious one is the monotonicity of the entropy $\nrjx$ resp.\ $\anrj$ in discrete time.
The second, more subtle one is the strict convexity of the minimization problem \eqref{eq:nmm},
see Proposition \ref{prp:convex} for details.
This guarantees not only uniqueness of the minimizer, but also efficient numerical solvability.
Convexity of \eqref{eq:nmm} is also the key ingredient for our stability analysis.

\subsection{Related discretizations from the literature}
The general idea to perform a fully Lagrangian discretization for problems of the type \eqref{eq:eq}--\eqref{eq:ic} 
dates back at least to the beautiful but apparently almost unknown paper by Russo \cite{russo}.
In fact, there is an earlier short communication by MacCamy and Socolovsky \cite{maccamy} 
in which a similar idea for use in one space dimension is sketched.
Later, Budd et al \cite{budd} used a variant of that discretization under the name ``moving mesh method'' 
to track asymptotically self-similar solutions of the porous medium equation 
--- which is \eqref{eq:eq} with $\prss(u)=u^m$ ($m>1$) and $V\equiv0$ --- in one space dimension numerically.
None of these schemes has been variational, i.e., the gradient flow structure did not play any role.

The first use of the interrelation between the Lagrangian and the gradient flow character 
for numerical purposes is apparently due to Gosse and Toscani \cite{GosTos}.
The method developed there, however, was limited to one space dimension,
where the Wasserstein metric between probability measures on the line
is isometrically equivalent to the $L^2$-distance between the corresponding inverse distribution functions;
the higher-dimensional situation is much more complicated, as discussed above.
Subsequently, the one-dimensional method has been refined and then applied to various further evolution equations,
like the Keller-Segel model \cite{blanchet}, 
the Hele-Shaw flow \cite{naldi}, 
the quantum drift diffusion equation \cite{bertram}, 
the $p$-Laplace equation \cite{agueh},
and even to the isentropic Euler equations \cite{westdickenberg}.
For some of those schemes, 
the discrete-to-continuous limit has been rigorously established by the second and the third author \cite{dde,dlss}.

In contrast, the first attempts to combine a fully Lagrangian discretization with a variational time integrator 
for solution of the \emph{multi-dimensional} problem \eqref{eq:eq}--\eqref{eq:ic} are only quite recent.
In fact, we are aware of only one fully Lagrangian scheme for solution of non-linear Fokker-Planck equations
that also respects the variational structure of the underlying gradient flow: 
this is the work by Carlier et al \cite{carlier}.
We emphasize that although their discretization approach has several similarities to our own
--- e.g., it produces a sequence of strictly convex minimization problems ---
the way they exploit the Lagrangian structure is conceptually different from 
the idea to iterate close-to-identity diffeomorphisms from a given ansatz space, as it is done here.
It should further be noted that some preliminary considerations in the same direction,
along with numerical experiments, have been made in the thesis of Roessler \cite{roessler}.

For completeness, we mention two further works that do not completely fit the context, but are still closely related:
first, Carrillo and Moll \cite{moll} developed a fully Lagrangian scheme for a non-local agregation equation; 
although that scheme is not truely variational, its definition is closely linked to the gradient flow structure.
Second, in a recent paper by Cavalletti et al \cite{westdickenberg2},
the isentropic Euler equations have been analyzed in the framework of Wasserstein gradient flows.
The construction used there can be interpreted as a fully Lagrangian variational solver.

\subsection{Plan of the paper}
The motivation and the details of our discretization approach are given in Section \ref{sct:discrete} below.
The rest of the paper is divided into two parts: 
an analytical one, in which we perform a rigorous study of well-posedness and consistency,
and a numerical one, in which we clarify practical aspects of the implementation and report results from experiments.
Our results in the analytical part are summarized in Section \ref{sct:results}.
Their respective proofs are given in Sections \ref{sct:proofs1} and \ref{sct:proofs2}.
Section \ref{sct:implement} is concerned with the implementation of the numerical scheme.
Numerical results are then given in Section \ref{sct:numerics}.

\section{Preliminaries}
\label{sct:discrete}
%
\subsection{Additional hypotheses and notations}
\label{sct:techhypos}
The class of non-negative densities with total mass $M$ is denoted by $\dens$.
The function $\basQ:\setR^d\to\setR$ is defined by 
\begin{align}
  \label{eq:basQ}
  \basQ(x)=\frac12|x|^2;
\end{align}
observe that $\nabla\basQ=\id$, the identity map on $\setR^d$.
For brevity, we introduce
\begin{align}
  \label{eq:vvelo}
  \velo(\T;\tm) := \uelo\big[(\T\circ\tm)_\#\rfd\big]\circ\T.
\end{align}
By $\maps$, 
we mean the class of (orientation-preserving) $C^1$-diffeomorphisms $\tm$ of $\overline\Omega$ onto itself that fix the corners.
Note that any $\tm\in\maps$ maps each of the $(d-1)$-dimensional boundaries of $\Omega$ onto itself,
that is, boundary points are only moved laterally.

\emph{Symmetry assumption:}
To avoid unnecessary technicalities, we restrict our considerations to solutions $u$ with a certain symmetry.
Specifically, we call a function $f:\Omega\to\setR$ \emph{symmetric-periodic}
if it admits a continuous extension $\tilde f:\setR^d\to\setR$ to the whole space
which is both even and $2$-periodic with respect to each coordinate.
We further say that $f$ is \emph{$C^k$-symmetric-periodic} if additionally $\tilde f\in C^k(\setR^d)$. 
For instance, any function of the form 
\[ f(x) = F(\cos \pi x_1,\ldots,\cos \pi x_d) \]
with a $C^k$-function $F:\setR^d\to\setR$ is $C^k$-symmetric-periodic.
Smooth symmetric-periodic solutions $u$ form an invariant class for the problem \eqref{eq:eq}--\eqref{eq:ic}
in the following sense:
assuming that both the potential $V$ and the initial datum $u^0$ are $C^\infty$-symmetric periodic,
it then follows by unique smooth solvability of \eqref{eq:eq}--\eqref{eq:ic}, see e.g. \cite{VazquezPME},
in combinaition with symmetry arguments that the corresponding solution $u_t$ is $C^\infty$-symmetric-periodic at any $t>0$.

\emph{Variational derivative and gradient:}
In order to avoid the use of the ambigous symbol $\delta F/\delta u$ for the variational derivative,
we introduce the following convention:
for a given functional $F$,
denote by $\dff F(u)[\varphi]$ the first variation of $F$ at $u$ in the direction of (the smooth function) $\varphi$.
Assuming that $\dff F(u)$ extends to a linear and continuous functional on $L^2(\theta)$, where $\theta$ is some given measure on $\Omega$,
let $\isom{L^2(\theta)}{\dff F(u)}\in L^2(\theta)$ be its Riesz-dual, i.e.,
\[ \intom \varphi(x)\left\{\isom{L^2(\theta)}{\dff F(u)}\right\}(x)\dd\theta(x) = \dff F(u)[\varphi]  \]
holds for all smooth $\varphi$.

\subsection{Analytical background --- two gradient flow structures, Lagrangian representation}
\label{sct:gf}
To motivate the variational discretization approach, 
we briefly discuss the gradient flow structure behind \eqref{eq:eq}--\eqref{eq:ic},
which is actually two-fold.


\subsubsection{First gradient flow structure}
One possibility to write \eqref{eq:eq}--\eqref{eq:ic} in variational form 
has been worked out in Otto's celebrated paper \cite{OttPME}.
Observe that $\uelo$ in \eqref{eq:velo} has the form  
\begin{align}
  \label{eq:W2flow}
  \uelo[u] = -\nabla\isom{L^2(\dn x)}{\dff\anrj(u)}, 
\end{align}
where $\anrj$ is the following \emph{relative entropy functional}
\begin{align}
  \label{eq:anrj_dde2D}
  \anrj(u) = \intom h(u(x))\dd x + \intom u(x)V(x)\dd x.
\end{align}
This makes solutions $t\mapsto u_t$ curves of steepest descent in the potential landscape of $\anrj$
on the space $\dens$ of densities with respect to the $L^2$-Wasserstein metric $\wass$.
We refer the reader to \cite{AGS,VilBook} for an introduction to Wasserstein metrics and associated gradient flows.

The main analytical insight from this gradient flow representation is that
$\anrj$ is a $\lambda$-convex functional in the metric $\wass$ \cite{McCann}.
More precisely, it is $\lambda$-convex along generalized geodesics in the sense of \cite{AGS}.
Here $\lambda\in\setR$ coincides with the modulus of convexity of $V$, i.e., $\nabla^2V\ge\lambda\eins$.
This makes qualitative features of the non-linear Fokker-Planck equation \eqref{eq:eq} 
accessible to the abstract theory of $\lambda$-contractive gradient flows,
with important consequences on the large-time behavior \cite{OttPME}.

\subsubsection{Second gradient flow structure}
An alternative variational formulation directly for the transport representation \eqref{eq:trapo}\&\eqref{eq:velo}
was given by Evans et al in \cite{Evans}, see also \cite{ALS,CarLis}.
Indeed, it turns out that \eqref{eq:evans} is again a gradient flow,
namely on the space of diffeormorphisms on $\Omega$
with respect to the Hilbertian structure induced by $L^2(\rfd)$.
That is,
\begin{align}
  \label{eq:L2flow}
  \uelo\big[\tm_\#\rfd\big]\circ\tm = -\isom{L^2(\rfd)}{\dff\nrjx(\tm)},
\end{align}
for the potential $\nrjx$ (on diffeomorphisms $\tm$) 
that is induced by the relative entropy $\anrj$ above:
\begin{align}
  \label{eq:newpot}
  \nrjx(\tm) = \anrj(\tm_\#\rfd)
  = \intom \left[h_*\left(\frac{\det\dff\tm}{\rfd}\right)+V\circ\tm\right]\dd\rfd .
\end{align}
The function $h_*:\setRpp\to\setR$ above is obtained from $h$ via $h_*(s)=sh(1/s)$.
Equality of the vector fields $\uelo$ in \eqref{eq:W2flow} and in \eqref{eq:L2flow} is verified by a short calculation,
see Lemma \ref{lem:2equivalent}.
%
\begin{rmk}
  In \cite{Evans}, the following more explicit representation of $\uelo[\tm_\#\rfd]$ has been derived:
  \begin{align}
    \label{eq:evans2}
    \uelo\big[\tm_\#\rfd\big]\circ\tm
    = - \frac1{\rfd}\div\left[\prss\left(\frac{\rfd}{\det \dff\tm}\right)\,(\dff\tm)^\#\right] - (\nabla V) \circ\tm.
  \end{align}
  Here the divergence $\div$ acts on the entries in each row of the matrix (yielding a column vector), 
  and $A^\#$ denotes the cofactor matrix of $A$, that is $(\det A)(A^{-1})^T=A^\#$.
  The calculations leading to \eqref{eq:evans2} are given in the proof of Lemma \ref{lem:el}.
\end{rmk}
We remark that, despite the fact that $\anrj$ is $\lambda$-convex with respect to the $L^2$-Wasserstein metric $\wass$,
the functional $\nrjx$ is poly-convex, but has no other useful convexity properties with respect to $L^2(\rfd)$.

\subsubsection{Comparison of the two structures}
Although both gradient flows represent the same Fokker-Planck equation,
they have different analytical features.
The first (``Eulerian'') approach is set on the nice space of positive densities, 
with a $\lambda$-convex functional,
but needs the complicated Wasserstein metric.
The second (``Lagrangian'') approach uses the convenient $L^2$-metric, 
but on the more complicated space of diffeomorphisms, 
and lies outside of the setting of contractive flows.
For our discretization purposes, we rely on the second formulation,
but we modify the variational problem ``in the direction of the Eulerian approach'' 
in order to profit from the convexity there.
This is made precise in Proposition \ref{prp:convex}.

\subsubsection{Lagrangian representation of the solution}

The following Lemma shows how to construct an adapted family of push-forward maps 
for a Lagrangian representation of a solution to the PDE problem \eqref{eq:eq}--\eqref{eq:ic}.
It is essentially an adaptation of the result from \cite[Section 3.2]{Evans} to our situation.
\begin{lem}
  \label{lem:Tfromu}
  Let $u:[0,\tau]\times\Omega\to\setRpp$ be a smooth symmetric-periodic classical solution to \eqref{eq:eq}.
  Define $\T:[0,\tau]\times\Omega\to\setR^d$ as the unique solution to the initial value problem
  \begin{align}
    \label{eq:Tfromu}
    \partial_\sigma\T_\sigma = \uelo[u_\sigma]\circ\T_\sigma, \quad \T_0=\id.
  \end{align}
  Then $u_\sigma=(\T_\sigma)_\#u_0$, 
  and consequently, $\T_{(\cdot)}$ is a solution to $\partial_\sigma\T_\sigma=\velo(\T_\sigma;\tm)$,
  where $\tm\in\maps$ is arbitrary with $\tm_\#\rfd=u_0$.
\end{lem}
\begin{proof}
  Let $\varphi:\Omega\to\setR$ be a $C^\infty$-symmetric-periodic test function.
  On the one hand, since 
  \[\partial_\sigma u_\sigma=\Delta P(u_\sigma)+\nabla\cdot(u_\sigma\nabla V)=-\nabla\cdot(u_\sigma\uelo[u_\sigma]),\]
  we have that
  \begin{align*}
    \frac{\dn}{\dn\sigma}\int\varphi u_\sigma\dd x 
    = \int \varphi\partial_\sigma u_\sigma\dd x
    = - \int \nabla\varphi\cdot\uelo[u_\sigma]\,u_\sigma\dd x.    
  \end{align*}
  On the other hand, recalling the definition of $\T_\sigma$ in \eqref{eq:Tfromu},
  \begin{align*}
    \frac{\dn}{\dn\sigma}\int\varphi\; (\T_\sigma)_\#u_0\dd x     
    & = \frac{\dn}{\dn\sigma}\int(\varphi\circ\T_\sigma)\,u_0\dd x     
      = -\int(\nabla\varphi\circ\T_\sigma)\cdot(\uelo[u_\sigma]\circ\T_\sigma)\,u_0\dd x \\
    & = -\int\nabla\varphi\cdot\uelo[u_\sigma]\,(\T_\sigma)_\#u_0\dd x.
  \end{align*}
  Subtracting those equations from each other yields
  \begin{align*}
    \frac{\dn}{\dn\sigma}\int\varphi\,\big[u_\sigma-(\T_\sigma)_\#u_0\big]\dd x     
    = -\int \uelo[u_\sigma]\cdot\nabla\varphi\,\big[u_\sigma-(\T_\sigma)_\#u_0\big]\dd x.
  \end{align*}
  Choosing $\varphi=2[u_\sigma-(\T_\sigma)_\#u_0]$ yields further that
  \begin{align*}
    \frac{\dn}{\dn\sigma}\int[u_\sigma-(\T_\sigma)_\#u_0]^2\dd x
    &= -\int \uelo(u_\sigma)\cdot\nabla\big[u_\sigma-(\T_\sigma)_\#u_0\big]^2\dd x \\
    &= \int \nabla\cdot \uelo(u_\sigma)\,\big[u_\sigma-(\T_\sigma)_\#u_0\big]^2\dd x
    \le A \int[u_\sigma-(\T_\sigma)_\#u_0]^2\dd x,
  \end{align*}
  where $A$ is a uniform bound on the smooth function
  \begin{align*}
    \nabla\cdot \uelo[u_\sigma] = \Delta h'(u_\sigma) + \Delta V.
  \end{align*}
  Since $(\T_0)_\#u_0=\id_\#u_0=u_0$,
  Gronwall's lemma yields that $(\T_\sigma)_\#u_0=u_\sigma$ for all $\sigma\in[0,\tau]$.
\end{proof}


\section{Discretization}

\subsection{Discretization in time}
Below, we motivate our choice of the time discretization
by discussing a \emph{semi}-discretization of \eqref{eq:eq}--\eqref{eq:ic}.
Specifically, consider approximations $(u_\tau^n)_{n\in\setN}$ and $(\tm_\tau^n)_{n\in\setN}$ for a given time step $\tau>0$
by means of the \emph{minimizing movement scheme}:
initially, choose $u_\tau^0\in\dens$ as approximation of $u^0$, and $\tm_\tau^0$ such that $(\tm_\tau^0)_\#\rfd=u_\tau^0$.
Then, define inductively
\begin{align}
  \label{eq:mm}
  u_\tau^n = \argmin_{u\in\dens}\left[\frac1{2\tau}\wass(u,u_\tau^{n-1})^2 + \anrj(u)\right],
  \quad
  \tm_\tau^n = \argmin_{\tm\in\maps}\left[\frac1{2\tau}\|\tm-\tm_\tau^{n-1}\|_{L^2(\rfd)}^2 + \nrjx(\tm)\right].
\end{align}
By some formal considerations 
--- which have been made rigorous under certain technical hypotheses in \cite{ALS} --- 
one concludes that the equivalence between the gradient flows 
is preserved under this time-discretization.
More precisely, each time step in \eqref{eq:mm} is moderated by a close-to-identity diffeomorphism $\T_\tau^n$ of $\Omega$ 
which is such that, simultaneously,
\begin{align*}
  u_\tau^n = (\T_\tau^n)_\#u_\tau^{n-1}
  \quad\text{and}\quad
  \tm_\tau^n = \T_\tau^n\circ\tm_\tau^{n-1}
\end{align*}
hold.
This $\T_\tau^n$ is determined as solution to the corresponding Euler-Lagrange equation 
\begin{align}
  \label{eq:1el}
  \T = \id+\tau\velo(\T;\tm_\tau^{n-1}),
\end{align}
with the abbreviation $\velo(\T;\tm)$ defined in \eqref{eq:vvelo}.
One concludes that $u_\tau^n=(\tm_\tau^n)_\#\rfd$ in each time step $n$,
and also that $\wass(u_\tau^n,u_\tau^{n-1})=\|\tm_\tau^n-\tm_\tau^{n-1}\|_{L^2}$.
In fact, $\T_\tau^n$ is an 
\emph{optimal transport map}\footnote{That is, $\T_\tau^n$ is a solution to \emph{Monge's problem}
  of finding a diffeomorphism $\T$ of $\Omega$ with $\T_\#u_\tau^{n-1}=u_\tau^n$
  such that $\int_\Omega|\T(x)-x|^2u_\tau^{n-1}(x)\dd x$ is minimal.}
from $u_\tau^{n-1}$ to $u_\tau^n$.
Thanks to the general theory \cite{AGS},
optimal transport maps $\T:\Omega\to\Omega$ are those which can be written as the gradient of a convex function, 
i.e., $\T=\nabla\phi$ for a suitable convex $\phi:\Omega\to\setR$.
\begin{rmk}
  We give a heuristic argument why solutions to the Euler-Lagrange equation \eqref{eq:1el} are optimal transport maps,
  at least for sufficiently small time steps $\tau$.
  This is not completely obvious since on the right-hand side of \eqref{eq:1el}, 
  we have the composition of the gradient vector field $\uelo$ with $\T$ itself;
  in general, such a composition is not necessarily a gradient again.
  Fortunately, \eqref{eq:1el} can be equivalently written in the form
  \begin{align*}
    \T^{-1}=\id-\tau\uelo\big[(\T\circ\tm_\tau^{n-1})_\#\rfd\big].
  \end{align*}
  Thus, if $\T$ is a solution, then its inverse $\T^{-1}$ is the gradient of some potential $f:\Omega\to\setR$.
  If this $f$ happens to be convex (which is a reasonable guess for sufficiently small $\tau>0$),
  then it follows by elementary convex analysis that $\T$ itself is the gradient of a convex function as well, 
  namely of the Legendre transformed $f^*$.

  For example, if $\uelo[u]=-\nabla V$ with $\nabla^2V\ge\lambda>-\infty$,
  then \eqref{eq:1el} has the unique solution $\T=(\id+\tau\nabla V)^{-1}$.
  Provided that $1+\tau\lambda>0$,
  this $\T$ is the gradient of the Legendre transform of the convex function $f(x)=\frac12|x|^2+\tau V(x)$.
\end{rmk}
One can therefore restrict the second minimization in \eqref{eq:mm} --- a priori ---
to maps of the form $\tm=\T\circ\tm_\tau^{n-1}$, where $\T$ is an optimal transport.
Thus, the second minimization problem in \eqref{eq:mm} reduces to
\begin{align}
  \label{eq:m}
  \tm_\tau^n=(\T_\tau^n)\circ\tm_\tau^{n-1}, \quad\text{with}\quad
  \T_\tau^n = \argmin_{\T\in\nabla\xspc}  
  \left[\frac1{2\tau}\|\T-\id\|_{L^2((\tm_\tau^{n-1})_\#\rfd)}^2 + \nrjx\big(\T\circ\tm_\tau^{n-1}\big)\right],
\end{align}
where, by abuse of notation, we wrote $\T\in\nabla\xspc$ to indicate that 
the minimization runs over all maps $\T:\Omega\to\Omega$ can be written in the form $\T=\nabla\phi$ 
for a suitable $\phi$ from the convex affine set 
\begin{align*}
  \xspc := 
  \big\{ \phi:\bar\Omega\to\setR\,;\,
  \text{$\phi$ is convex},\,
  \nabla\phi(\Omega)\subseteq\Omega
  \big\}.
\end{align*}
For further reference, we introduce the ``regular'' subset
\begin{align*}
  \xspc^+ = \left\{ \phi\in\xspc\cap C^2(\overline\Omega)\,;\,
  \inf_\Omega\det\nabla^2\phi>0,\,
  \phi-\basQ\text{ is symmetric-periodic}
  \right\},
\end{align*}
recalling the definition of $\basQ$ in \eqref{eq:basQ}.
Notice that a consequence of $\phi-\basQ$ being symmetric-periodic is that
\begin{align*}
  \nrml\cdot(\nabla\phi-\id)\equiv0\quad \text{on $\partial\Omega$},
\end{align*}
which in turn implies that $\nabla\phi$ maps onto $\Omega$,
and that points on the four boundary segments only move laterally (and the corners stay fixed) under $\nabla\phi$.
In the same way, we understand $\T\in\nabla\xspc^+$; 
notice that such a $\T$ is an optimal transport map and lies in $\maps$.
 
The reformulation \eqref{eq:m} does not only reduce the complexity of the variational problem \eqref{eq:mm}
--- by moving from diffeomorphisms to scalar functions ---
but restores $\lambda$-convexity:
\begin{prp}
  \label{prp:convex}
  For each $\tm\in\maps$, the functional $\T\mapsto\nrjx(\T\circ\tm)$ 
  is $\lambda$-convex on $\nabla\xspc$ in the usual (flat) sense,
  with $\lambda$ being $V$'s modulus of convexity.  
\end{prp}
The proof is given in Section~\ref{subsec:genconv}.

\subsection{Discretization in space}
Spatial discretization is achieved by restriction of the potentials $\phi$ for the close-to-identity transformations $\T_\tau^n$
to a finite-dimensional convex affine set $\xspc_\theh \subset\xspc$ of ansatz functions,
where $\theh$ is an indicator for the spatial resolution.
Specifically, we assume that 
a system $\{\bas_\ki\}_{\ki\in\I}$ of $C^\infty$-symmetric-periodic ansatz functions $\bas_\ki:\Omega\to\setR$ is given,
that forms an orthogonal basis of $L^2(\dn x)$.
To each $\theh>0$, a subset $\I_\theh\subset\I$ of indices $\ki$ is associated 
such that $\I_{\theh'}\subseteq\I_\theh$ whenever $\theh'\ge\theh>0$.
Now introduce
\begin{align*}
  \txspc_\theh = \left\{ \psi =  \sum_{\ki\in\I_\theh}z_\ki\bas_\ki\,;\, z_\ki\in\setR\right\},
\end{align*}
the set of linear combinations of elements in $\{\nabla\bas_\ki\}_{\ki\in\I_\theh}$.
From that, we define
\begin{align*}
  \xspc_\theh = \left\{ \phi = \basQ + \psi\,;\,\phi\in\txspc_\theh\right\} \cap \xspc,
\end{align*}
and accordingly $\xspc_\theh^+$.
\begin{rmk}
  A very natural choice for the $\bas_\ki$ are Fourier modes,
  so that $\txspc_\theh$ is a space of Fourier polynomials.
  This is what we actually use for our numerical experiments, cf.\ Section~\ref{sct:proofs2} for details.  
\end{rmk}
For simplicity, 
we combine the time step $\tau>0$ and the indicator $\theh$ into a single discretization parameter $\Delta=(\tau;\theh)$.
Now \eqref{eq:m} is replaced by
\begin{align}
  \label{eq:md}
  \tm_\Delta^n=\T_\Delta^n\circ\tm_\Delta^{n-1}, \quad\text{with}\quad
  \T_\Delta^n = \argmin_{\T\in\nabla\xspc^+_\theh} 
  \left[\frac1{2\tau}\|\T-\id\|_{L^2((\tm_\Delta^{n-1})_\#\rfd)}^2 + \nrjx\big(\T\circ\tm_\Delta^{n-1}\big)\right].
\end{align}
Finally, in accordance with \eqref{eq:concat}, 
the approximate Lagrangian map $\tm_\Delta^n:\Omega\to\Omega$ at time step $n$ 
is now obtained as concatenation of the $n$ optimal transports $\T_\Delta^1,\ldots,\T_\Delta^n\in\nabla\xspc^+_\theh$. 
The associated ``discretized'' densities are then given by
\begin{align}
  \label{eq:uDelta}
  u_\Delta = \big( u_\Delta^n \big)_{n=0,1,2,\ldots}, 
  \quad
  u_\Delta^n=(\tm_\Delta^n)_\#\rfd .
\end{align}

\section{Summary of analytical results}
\label{sct:results}
The fully discrete scheme is well-posed:
\begin{thm}
  \label{thm:well-posed}
  Assume that $s\mapsto s^{d/2-1}h_*(s)$ is not integrable at $s=0$, or, 
  equivalently, that $r\mapsto r^{-(d/2+2)}h(r)$ is not integrable at $r=+\infty$.
  Let the time step $\tau>0$ be sufficiently small to verify $\lambda+\tau^{-1}>0$.

  Then the minimization problems in \eqref{eq:md} 
  are convex in each step and can be uniquely solved inductively.
  The $n$th iterative map $\T_\Delta^n$ belongs to $\xspc_\theh^+$ 
  and is the unique solution to the following system of Euler-Lagrange equations,
  \begin{align}
    \label{eq:deuler}
    \left\langle\frac{\T-\id}\tau-\velo(\T;\tm_\Delta^{n-1}),\nabla\psi\right\rangle_{L^2((\tm_\Delta^{n-1})_\#\rfd)} = 0
    \qquad \text{for all $\psi\in\txspc_\theh$}.
  \end{align}
\end{thm}
The proof is given in Section~\ref{sct:wellposed}. 

The next result is concerned with the continuous limit $\tau\to0$ and $\theh\to0$.
For a special choice of Fourier ansatz functions $\bas_\ki$  --- cf.\ Section~\ref{sct:proofs2} for details ---
we have the following convergence result.
\begin{thm}
  \label{thm:consist}
  Let $u:[0,T]\times\Omega\to\setRp$ 
  be a classical symmetric-periodic solution to \eqref{eq:eq}--\eqref{eq:ic}.
  Choose $\tm_\Delta^0\in \maps$ such that $(\tm_\Delta^0)_\#\rfd = u^0$,
  and define the discrete approximation $u_\Delta$ by means of the scheme \eqref{eq:m}, with \eqref{eq:uDelta}.
  We make the following stability assumption:
  the densities $u_\Delta^n$ are are $\Delta$-independently bounded in $C^5(\Omega)$, and positively bounded from below,
  that is, there are constants $B,\beta>0$ such that
  \begin{align}
    \label{eq:smoothhypo}
    \|u_\Delta^n\|_{C^5}\le B, \quad u_\beta^n\ge\beta \quad \text{for all $\tau>0$ and $n\in\setN$ with $n\tau\le T$}.
  \end{align}
  Then the discrete solutions $u_\Delta$ converge to the smooth solution $u$ in Wasserstein,
  \begin{align}
    \label{eq:theresult}
    \wass(u_{n\tau},u_\Delta^n) =\bigO(\tau+\rho_\theh),
  \end{align}
  where $\rho_{(\cdot)}:(0,1)\to\setRpp$ is an increasing function with $\rho_\theh\to0$ as $\theh\to0$,
  see Lemma \ref{lem:projestimate} for details.
\end{thm}
The proof is given in Section \ref{sct:consist}.
\begin{rmk}
  \begin{itemize}
  \item Note that the hypotheses and conclusions are formulated in term of densities,
    not in terms of the transport maps by which they are induced.
    With additional technical effort, an approximation of the transport should be possible, 
    using the techniques from \cite{ALS}.
  \item Since we require \eqref{eq:smoothhypo} in particular at the initial time, 
    it follows from the regularity theory for filtration equations that $u$ is $C^\infty$-symmetric-periodic. 
  \item The assumption of $\Delta$-uniform $C^5$-bounds on the discrete approximations 
    is certainly a strong hypotheses to conclude convergence in the Wasserstein metric.
    Unfortunately, our only stability result is the one related to the convexity of the minimization problem \eqref{eq:md},
    which provides an estimate in the Wasserstein metric, see Lemma \ref{lem:convex}, but no better.
    A much higher degree of regularity of the densities $u_\Delta^n$ is needed to control 
    the approximation error for the velocity field by means of Lemma \ref{lem:vLip} and Lemma \ref{lem:projestimate}.
  \item In any case, by standard interpolation arguments, 
    one easily deduces approximation in better spaces, 
    like
    \begin{align*}
      \|u_{n\tau}-u_\Delta^n\|_{C^4} \to 0
    \end{align*}
    as $\tau\to0$ and $\theh\to0$ simultaneously.
    The estimate on the approximation error above, however, will not be linear anymore in $\tau$ and $\rho_\theh$.
  \end{itemize}
\end{rmk}

%

\section{Proof of well-posedness (Theorem \ref{thm:well-posed})}
\label{sct:proofs1}
%

\subsection{Generalized convexity}
\label{subsec:genconv}
As a preliminary result of somewhat independent interest, 
we show the convexity property stated in Proposition \ref{prp:convex}.
First, we give the formal argument, which is short and instructive, following \cite[Theorem 2.2]{McCann};
afterwards, we sketch the rigorous argument.
Here ``formal'' just means that we consider $\T\mapsto\nrjx(\T\circ\tm)$ as a functional on the regular space $\nabla\xspc^+$.
The extension of the argument to all of $\nabla\xspc$ requires additional technical effort.
\begin{proof}[Proof of Proposition \ref{prp:convex} for regular arguments]
  Fix $\tm\in\maps$.
  Let $\T_0,\T_1\in\nabla\xspc^+$ be given.
  Then $\T_s:=(1-s)\T_0+s\T_1$ belongs to $\nabla\xspc^+$ as well, for every $s\in[0,1]$,
  Therefore, $\det\dff\T_s$ is a continuous and strictly positive function on $\overline\Omega$,
  and so the following calculations are justified:
  \begin{align*}
    \nrjx(\T_s\circ\tm) 
    &= \intom h_*\left(\frac{\det\dff(\T_s\circ\tm)}{\rfd}\right)\dd\rfd + \intom V\circ(\T_s\circ\tm)\dd\rfd \\
    &= \intom h_*\left(\frac{(\det\dff\T_s)\circ\tm}{\rfd/\det\dff\tm}\right)\dd\rfd + \intom (V\circ\T_s)\circ\tm\dd\rfd\\
    &= \intom h_*\left(\frac{\det\dff\T_s}{\tm_\#\rfd}\right)\dd\tm_\#\rfd + \intom V\circ\T_s\dd\tm_\#\rfd.
  \end{align*}
  Since $\T_0=\nabla\phi_0$ and $\T_1=\nabla\phi_1$ for convex functions $\phi_0,\phi_1\in\xspc^+$,
  it follows that $\T_s=\nabla\phi_s$ for the convex function $\phi_s=(1-s)\phi_0+s\phi_1\in\xspc^+$.
  We have
  \begin{align*}
    \det\dff\T_s = \det\big((1-s)\nabla^2\phi_0+s\nabla^2\phi_1\big)
    = \det\nabla^2\phi_0\,\det\big((1-s)\eins+sM\big),
  \end{align*}
  with the matrix $M=(\nabla^2\phi_0)^{-1}\nabla^2\phi_1$, 
  which is symmetric and positive definite, and hence diagonalizable with only positive eigenvalues, at every $x\in\Omega$.
  The eigenvalues of the convex combination $(1-s)\eins+sM$ are the respective convex combinations of the eigenvalues of $M$ with one.
  Since the determinant is the product of the eigenvalues, 
  it follows --- see, for instance, \cite{hadwiger} for a proof --- that
  \begin{align*}
    s\mapsto\big(\det\dff\T_s\big)^{1/d}
  \end{align*}
  is a concave map.
  Since $z\mapsto h_*(az)$ is a convex and decreasing map for each $a>0$,
  it follows that $s\mapsto h_*\big(\det\dff\T_s(x)/\tm_\#\rfd(x)\big)$ is a convex map as well,
  for each $x\in\Omega$, proving convexity of the first integral above.
  Concerning the second integral, it suffices to observe that, for each fixed $x\in\Omega$, 
  the map $s\mapsto V\circ\T_s(x)$ inherits the $\lambda$-convexity of $V$.
\end{proof}
In the general case with $\T\in\nabla\xspc$, 
a rigorous argument can be obtained along the lines of \cite[Proposition 9.3.9]{AGS}.
There, it is proven that for any given optimal transport maps $\T_0,\T_1\in\nabla\xspc$, 
the functional $\anrj$ is $\lambda$-convex 
along the generalized geodesic $s\mapsto u_s:=((1-s)\T_0+s\T_1)_\#(\tm_\#\rfd)$ in the $L^2$-Wasserstein space.
By definition --- see \cite[Definition 9.2.2]{AGS} --- this means that 
the scalar function $s\mapsto\anrj(u_s)$ is $\lambda$-convex on $[0,1]$.
Since 
\[ \anrj(u_s)=\nrjx\big( ((1-s)\T_0+s\T_1)\circ\tm \big) \]
by definition \eqref{eq:newpot} of $\nrjx$, the proposed convexity statement follows.
\begin{cor}
  \label{cor:lambda}
  For any $\tm\in\maps$ and $\T^0,\T^1\in\nabla\xspc^+$,
  \begin{align}
    \label{eq:lammonotone}
    -\langle\T^1-\T^0,\velo(\T^1;\tm)-\velo(\T^0;\tm)\rangle_{L^2(\tm_\#\rfd)}
    \ge \lambda\|\T^1-\T^0\|_{L^2(\tm_\#\rfd)}^2.
  \end{align}
\end{cor}
\begin{proof}
  By Proposition \ref{prp:convex} above,
  the functional $\T\mapsto\nrjx(\T\circ\tm)$ is $\lambda$-convex on $\nabla\xspc$ with respect to $L^2(\tm_\#\rfd)$,
  and by Lemma \ref{lem:2equivalent}, its gradient in that space is given by $-\velo(\T;\tm)$. 
  Therefore, a Taylor expansion yields that 
    \begin{align*}
    \nrjx(\T^1\circ\tm)\ge\nrjx(\T^0\circ\tm)
    -\langle\T^1-\T^0,\velo(\T^0;\tm)\rangle_{L^2(\tm_\#\rfd)}
    +\frac\lambda2\|\T^1-\T^0\|_{L^2(\tm_\#\rfd)}^2.
  \end{align*}
  The same is true with the roles of $\T^0$ and $\T^1$ exchanged.
  Addition of both inequalities provides \eqref{eq:lammonotone}.
\end{proof}

\subsection{Well-posedness of the time iteration}
\label{sct:wellposed}
We are now in the position to prove that the minimization problems in \eqref{eq:md} 
can be solve iteratively w.r.t.\ $n=1,2,\ldots$. 
\begin{proof}[Proof of Theorem \ref{thm:well-posed}]
  Fix some $\tm\in\maps$.
  We consider the functional $\Phi:\nabla\xspc_\theh\to\setR\cup\{+\infty\}$ given by
  \begin{align*}
    \Phi(\T)= \frac1{2\tau}\|\T-\id\|_{L^2(\tm_\#\rfd)}^2+\nrjx(\T\circ\tm).
  \end{align*}
  Clearly, $\Phi(\T)$ is finite for all $\T\in\nabla\xspc_\theh^+$, 
  but we cannot rule out $\Phi(\T)=+\infty$ for general $\T\in\nabla\xspc_\theh$.

  Now, since $\xspc_\theh$ is a convex subset of $\xspc$, also $\nabla\xspc_\theh$ is a convex subset of $\nabla\xspc$.
  Clearly, the finite-dimensional set $\xspc_\theh$ is closed and bounded, and therefore compact.
  Thanks to Proposition \ref{prp:convex}, the restriction to $\nabla\xspc_\theh$ of the map $\T\mapsto\nrjx(\T\circ\tm)$ is $\lambda$-convex.
  Clearly, the map $\T\mapsto\frac12\|\T-\id\|_{L^2(\tm_\#\rfd)}^2$ is $1$-convex on $\nabla\xspc_\theh$,
  and therefore, $\Phi$ is a $(\lambda+\tau^{-1})$-convex functional.
  In particular, $\Phi$ is uniformly convex for $0<\tau<(-\lambda)^{-1}$ if $\lambda<0$,
  and for arbitrary $\tau>0$ if $\lambda\ge0$, respectively, independently of $\tm$.

  Observe further that $\Phi$ is bounded from below, since $h_*$ is non-negative and $V$ is smooth.
  Consequently, $\Phi$ possesses a unique minimizer $\T^*\in\nabla\xspc_\theh$.
  It remains to verify that $\T^*\in\nabla\xspc^+_\theh$.

  Thanks to the properties of the basis functions $\bas_\ki$, 
  we have that $\T^*\in C^3(\overline\Omega)$,
  and that $\T^*-\id$ is symmetric-periodic.
  It only remains to verify that $\dff\T^*$ is uniformly positive definite on $\Omega$.
  For that, it suffices to show that $f:\Omega\to\setR$ with $f(x)=\det\dff\T^*(x)$ is a strictly positive function.
  By the properties above, $f\in C^2(\overline\Omega)$ is symmetric-periodic.
  Further, since $\T^*\in\nabla\xspc$ is the gradient of a convex function, we immediately have that $f$ is non-negative.
  Towards a contradiction, assume that a positive lower bound does not exist.
  By continuity, there must exist a point $\bar x\in\overline\Omega$ such that $f(\bar x)=0$.
  Next, non-negativity of $f$ implies that $\nabla f(\bar x)=0$;
  this is true even at the boundary and in the corners of $\Omega$, thanks to $f$ being symmetric-periodic.
  From here, $C^2$-smoothness of $f$ implies that
  \begin{align*}
    f(x) \le c|x-\bar x|^2 \quad\text{for all $x\in\ball_r(\bar x)\cap\Omega$},
  \end{align*}
  for suitable constants $c>0$ and $r>0$.
  Let us show that this implies $\Phi(\T^*)=+\infty$, contradicting minimality of $\T^*$.
  As a $C^2$-function, $f$ is bounded on $\Omega$, 
  hence $a:=\inf\tm_\#\rfd=\inf(\rfd/f)>0$.
  Since furthermore $h_*$ is a positive and decreasing function, 
  we have the following estimate:
  \begin{align*}
    \intom h_*\left(\frac{\det\dff\T^*}{\tm_\#\rfd}(x)\right)\dd\tm_\#\rfd(x)
    &\ge \int_{\ball_r(\bar x)\cap\Omega} h_*\left(\frac{c|x-\bar x|^2}{a}\right)\,a\dd x \\
    &\ge a\frac{|\sphere^{d-1}|}{2^d}\int_0^r h_*\left(\frac{c}a\rho^2\right)\,\rho^{d-1}\dd\rho
      = a\frac{|\sphere^{d-1}|}{2^{d+1}} \int_0^r h_*(c\eta/a)\,\eta^{d/2-1}\dd\eta.
  \end{align*}
  By hypothesis on the non-integrability of $s\mapsto s^{d/2-1}h_*(s)$ at the origin, the last integral is infinite.
\end{proof}

\section{Proof of weak convergence (Theorem \ref{thm:consist})}
\label{sct:proofs2}
\subsection{Basis functions and their properties}
For the proof of Theorem \ref{thm:consist}, 
we assume that the family of basis functions $\basis=\{\bas_\ki\}_{\ki\in\I}$ 
is given by a tensor product of of Fourier modes.
More precisely, we take the index set $\I=\setN_0^d\setminus\{0\}$
and define the ansatz function $\bas_\ki$ for the index $\ki=(k_1,\ldots,k_d)$ by
\begin{align}
  \label{eq:fbasis}
  \bas_\ki(x) = \frac{2^{d/2}}{\pi|\ki|_2}\cos(k_1\pi x_1)\cdots\cos(k_d\pi x_d). 
\end{align}
For the finite index sets $\I_\theh$, 
we take 
\[\I_\theh = \left\{\ki\in\I : |\ki|_\infty\le K\right\}\quad 
\text{with $K\in\setN$ such that $K\theh\in[1;2)$}.\]
\begin{lem}
  The family $\basis=\{\bas_\ki\}_{\ki\in\I}$ is a complete orthogonal system in $L^2(\dn x)$.
  The derived family $\nabla\basis=\{\nabla\bas_\ki\}_{\ki\in\I}$ is an orthonormal system in $L^2(\dn x)$.
\end{lem}
\begin{proof}
  The first claim is a classical result.
  It can be proven using the two facts that
  \begin{itemize}
  \item every function $f\in L^2(\dn x)$ can be approximated (in norm) by sums of functions of the form
    $g(x_1,x_2,\ldots,x_d)=g_1(x_1)g_2(x_2)\cdots g_d(x_d)$, and
  \item the cosines-modes $\cos(m\pi x)$ for $m=1,2,\ldots$ form a complete orthogonal system in $L^2((0,1))$.
  \end{itemize}
  To prove the second part, fix indices $\ki,\li\in\I$.
  Since $\bas_\ki$ and $\bas_\li$ are symmetric-periodic,
  the following integration by parts is justified:
  \begin{align*}
    \langle\nabla\bas_\ki,\nabla\bas_\li\rangle_{L^2(\dn x)}
     & = \intom \nabla\bas_\ki(x)\cdot\nabla\bas_\li(x)\dd x 
       = -\intom \bas_\ki(x)\,\Delta\bas_\li(x)\dd x 
       = \pi^2|\li|_2^2\intom\bas_\ki(x)\bas_\li(x)\dd x \\
     & = 2^d\frac{|\li|_2}{|\ki|_2}
      \left(\int_0^1\cos(k_1\pi x_1)\cos(\ell_1\pi x_1)\dd x_1\right)
      \cdots
      \left(\int_0^1\cos(k_d\pi x_d)\cos(\ell_d\pi x_d)\dd x_d\right) \\
      & = \delta_{\ki,\li}
  \end{align*}
  For the last step, we have used the orthonormality of the $\cos$-functions on the interval $[0,\pi]$.
\end{proof}
Thanks to the orthonormality, 
the $L^2(\dn x)$-orthogonal projection $\proj_\theh$ of vector fields $v:\Omega\to\setR^d$ 
onto the span of $\nabla\basis_\theh$ has the following easy representation:
\begin{align*}
  \proj_\theh[v] = \sum_{\ki\in\I_\theh} \langle v,\nabla\bas_\ki\rangle_{L^2(\dn x)}\nabla\bas_\ki.
\end{align*}
%
\begin{lem}
  \label{lem:projestimate}
  There exists an increasing function $\rho_{(\cdot)}:(0,1)\to\setRpp$ with $\rho_\theh\to0$ as $\theh\to0$
  such that the following is true.
  Suppose that $v:\Omega\to\setR^d$ is a vector field, 
  which admits a decomposition in the form $v=\nabla\psi+\zeta$, 
  where $\psi:\Omega\to\setR$ is a $C^4$-symmetric-periodic function,
  and $\zeta$ is a $C^3$-symmetric-periodic remainder.
  Then
  \begin{align}
    \label{eq:projestimate}
    \left\|v-\proj_\theh[v]\right\|_{C^2}\le \|\nabla\psi\|_{C^3}\rho_\theh + A\|\zeta\|_{C^3}.
  \end{align}
\end{lem}
\begin{rmk}
  The proof provides a rough upper bound on $\rho$ in the form $\rho_\theh\propto\delta^{1/2}$.
\end{rmk}
\begin{proof}
  Thanks to the linearity of $\proj_\theh$, we have that
  \begin{align*}
    \left\|v-\proj_\theh[v]\right\|_{C^2} 
    \le \left\|\nabla\psi-\proj_\theh[\nabla\psi]\right\|_{C^2} + \left\|\zeta-\proj_\theh[\zeta]\right\|_{C^2},
  \end{align*}
  and thus can estimate the terms related to $\nabla\psi$ and to $\zeta$ separately.
  For $\zeta$, we simply use 
  the $\theh$-uniform continuity of $\proj_\theh$ as a map from $C^3(\Omega)$ to $C^2(\Omega)$,
  \begin{align*}
     \left\|\zeta-\proj_\theh[\zeta]\right\|_{C^2}
    \le \left\|\zeta\right\|_{C^2}+\left\|\proj_\theh[\zeta]\right\|_{C^2}
    \le \left\|\zeta\right\|_{C^3} + A' \left\|\zeta\right\|_{C^3}.
  \end{align*}
  For the other estimate, we may assume without loss of generality that $\psi$ is of vanishing mean.
  Define the Fourier representation $\psi_\theh$ of $\psi$ with accuracy $\theh$ 
  by
  \begin{align}
    \label{eq:phidevelop}
    \psi_\theh = \sum_{\ki\in\I_\theh} \|\ki\|_2^2\langle\psi,\bas_\ki\rangle_{L^2(\dn x)}\bas_\ki.
  \end{align}
  One has that $\nabla\psi_\theh=\proj_\theh[\nabla\psi]$ since
  \begin{align*}
    \nabla\psi_\theh 
    = \sum_{\ki\in\I_\theh} \langle\psi,\|\ki\|_2^2\bas_\ki\rangle_{L^2(\dn x)}\nabla\bas_\ki 
    = \sum_{\ki\in\I_\theh} \langle\psi,-\Delta\bas_\ki\rangle_{L^2(\dn x)}\nabla\bas_\ki 
    = \sum_{\ki\in\I_\theh} \langle\nabla\psi,\nabla\bas_\ki\rangle_{L^2(\dn x)}\nabla\bas_\ki 
      = \proj_\theh[\nabla\psi].
  \end{align*}
  Therefore, the claim clearly follows if we can show that
  \begin{align}
    \label{eq:fouest}
    \|\psi-\psi_\theh\|_{C^3} \le C\rho_\theh\|\psi\|_{C^4}.
  \end{align}
  Clearly, the map $\psi\mapsto\psi_\theh$ is linear.
  It is easily seen that any $C^4$-symmetric-periodic function $\psi$ can be approximated in $C^4(\Omega)$
  by sums of functions $g$ that factorize in the coordinates, i.e., $g(x_1,x_2,\ldots,x_d)=g_1(x_1)g_2(x_2)\cdots g_d(x_d)$,
  with each $g_j\in C^4(\setR)$ being even and $2$-periodic.
  Hence it suffices to prove \eqref{eq:fouest} for such product functions $g$ in place of $\psi$.
  Performing the projection \eqref{eq:phidevelop} on $g$ yields again a factorizing function:
  \begin{align*}
    g_\theh(x_1,\ldots,x_d) = g_{1,\theh}(x_1)\cdots g_{d,\theh}(x_d),
  \end{align*}
  where the individual terms in the factorization are given in the classical way,
  \begin{align*}
    g_{j,\theh}(y) = \int_0^1 \Gamma_K(y,z)g_j(z)\dd z
    \quad\text{with the kernel}\quad
    \Gamma_K(y,z)=\sum_{k=1}^K \cos(k\pi y)\cos(k\pi z).
  \end{align*}
  Note that our particular choice of the index set $\I_\theh$ is important here.
  Comparing corresponding partial derivatives on both side of \eqref{eq:fouest}, 
  it is now easily seen that the proof of that estimate reduces to showing
  the following relation in one space dimension:
  \begin{align}
    \label{eq:hesthaven}
    \sup_y\left|f^{(\ell)}(y)-\int_0^1\partial_y^\ell\Gamma_K(y,z)f(z)\dd z\right| \le c_K\max_y\left|f^{(\ell+1)}(y)\right|,
  \end{align}
  for $\ell=1,2,3$ and all $f\in C^4(\setR)$ that are $2$-periodic and even,
  with a constant $c_K$ that tends to zero as $K\to\infty$.  
  This is covered by standard estimates on the approximation of smooth functions by Fourier series. 
  For instance, \cite[Theorem 2.12]{JanSHesthaven:2007vy} yields \eqref{eq:hesthaven} with $c_K\propto K^{-1/2}$.
\end{proof}
\begin{lem} 
  Let $\tilde u:[0,\tau]\times\Omega\to\setRp$, 
  $\tilde u_\sigma = \tilde u(\sigma;\;\cdot\;)$, 
  be a smooth symmetric-periodic classical solution to \eqref{eq:eq}--\eqref{eq:ic} 
  with $u_0=u_\Delta^{n-1}=(\tm_\Delta^{n-1})_\#\rfd$ and $\tilde\T$ the flow map from Lemma~\ref{lem:Tfromu}.  
  Under the hypotheses of Theorem~\ref{thm:consist}, there is a constant $C>0$, independent of $\tau$ and $n$,  
  such that
  \begin{align}
    \label{eq:Ttau2}
    \left\|\frac{\tilde{\T}_\tau-\proj_\theh[\tilde{\T}_\tau]}\tau\right\|_{C^2} 
    \le C(\tau+\rho_\theh).
  \end{align}
\end{lem}
\begin{proof}
  Thanks to the regularity hypothesis \eqref{eq:smoothhypo}, assuming that $\tau$ has been chosen sufficiently small, one has the bounds
  \[\|\tilde u_\sigma\|_{C^5}\le 2B, \quad \tilde u_\sigma\ge\beta/2.\]
  Hence the associated $\sigma$-dependent vector field 
  \begin{align}
    \label{eq:dummy1}
    \tilde{\velo}_\sigma=\uelo[\tilde u_\sigma] = -\nabla\big(h'(\tilde u_\sigma)+V\big) 
  \end{align} 
  is bounded in $C^4(\Omega)$, uniformly in $\sigma\in[0,\tau]$.
  Consequently, also its flow map $\tilde{\T}$, given by
  \begin{align}
    \label{eq:tildeflow}
    \partial_\sigma\tilde{\T}_\sigma=\tilde{\velo}_\sigma\circ\tilde{\T}_\sigma,
    \quad \tilde{\T}_0= \id, 
  \end{align}
  is $\sigma$-uniformly bounded in $C^4(\Omega)$.
  Lemma \ref{lem:Tfromu} guarantees that $(\tilde{\T}_\tau)_\#u_\Delta^{n-1} = \tilde u_\tau$,
  and that
  \begin{align}
    \label{eq:helpy}
    \tilde{\uelo}_\tau\circ\tilde{\T}_\tau = \velo(\tilde{\T}_\tau;\tm_\tau^{n-1}).
  \end{align}
  Writing out \eqref{eq:Tfromu} in integral form yields
  \begin{align}
    \label{eq:Ttau}
    \tilde{\T}_\tau 
    = \id + \int_0^\tau \tilde\uelo_\sigma \circ\tilde{\T}_\sigma\dd\sigma 
    = \id + \tau \tilde\uelo_\tau\circ\tilde{\T}_\tau + \tau^2\zeta_1
  \end{align}
  with the remainder term
  \begin{align*}
    \zeta_1
    &= \frac1\tau\strokedint_0^\tau\big(\tilde\uelo_\sigma\circ\tilde{\T}_\sigma-\tilde\uelo_\tau\circ\tilde{\T}_\tau\big)\dd\sigma \\
    &= -\frac1\tau\strokedint_0^\tau\int_{\sigma}^\tau\frac{\dn}{\dn\sigma'}\tilde\uelo_{\sigma'}\circ\tilde{\T}_{\sigma'}\dd\sigma'\dd\sigma \\
    &= -\frac1\tau\strokedint_0^\tau\int_{\sigma}^\tau
      \big(\dff\tilde\uelo_{\sigma'}\cdot\tilde\uelo_{\sigma'}+\partial_{\sigma'}\tilde\uelo_{\sigma'}\big)\circ\tilde{\T}_{\sigma'}\dd {\sigma'}\dd\sigma,
      %
  \end{align*}
  which can be estimated as follows:
  \begin{align}
    \label{eq:techest1}
    \|\zeta_1\|_{C^3} 
    \le A\sup_{0<\sigma<\tau}\left[\big(\|\tilde{\velo}_\sigma\|_{C^4}\|\tilde{\velo}_\sigma\|_{C^3}+\|\tilde\velo_\sigma\|_{C^3}\big)
    \big(1+\|\tilde{\T}_\sigma\|_{C^3}^3\big)\right]
    \le B'
  \end{align}
  where $A$ is a combinatorial factor that originates from the repeated application of the chain rule,
  and $B'$ can --- in principle --- be calculated from $B$ and $\beta$.
  Moreover, we have that
  \begin{align}
    \label{eq:TtauX}
    \tilde\uelo_\tau\circ\tilde\T_\tau = \tilde\uelo_\tau + \tau\zeta_2,
  \end{align}
  with remainder
  \begin{align*}
    \zeta_2 =\strokedint_0^\tau\big(\dff\tilde\uelo_\tau\cdot\tilde\uelo_\sigma\big)\circ\tilde\T_\sigma\dd\sigma,    
  \end{align*}
  that is controlled in the form
  \begin{align}
    \label{eq:techest2}
    \|\zeta_2\|_{C^3} 
    \le A\sup_{0<\sigma<\tau}\left[\|\tilde\uelo_\tau\|_{C^4}\|\tilde\uelo_\sigma\|_{C^3}\big(1+\|\tilde\T_\sigma\|_{C^3}^3\big)\right] \le B'
  \end{align}
  We conclude from \eqref{eq:Ttau}\&\eqref{eq:TtauX} that
  \[ \frac{\tilde{\T}_\tau-\id}\tau = \nabla\psi + \zeta, \]
  where $\psi=-[h'(\tilde u_\sigma)+V]$, see \eqref{eq:dummy1}, 
  and $\zeta=\tau(\zeta_1+\zeta_2)$ is controlled by means of \eqref{eq:techest1}\&\eqref{eq:techest2}.
  With the help of Lemma \ref{lem:projestimate}, we thus obtain the desired estimate \eqref{eq:Ttau2}.
\end{proof}

\subsection{Weak stability}
\begin{lem}
  \label{lem:vLip} 
  Assume that $\T^1,\T^2\in C^2(\Omega)$, and write $\hat u=\tm_\#\rfd\in C^3(\Omega)$ with some $\tm\in\maps$.
  Assume further that $\lambda\eins\le\dff\T^j\le \lambda^{-1}\eins$ for $j=1,2$, and $\lambda\le\hat u\le\lambda^{-1}$ 
  with some $\lambda>0$ uniformly on $\Omega$.
  Then there exists a constant $L$ 
  --- expressible in terms of $\lambda$, $\|V\|_{C^2}$ and the norms $\|\T^1\|_{C^2}$, $\|\T^2\|_{C^2}$ and $\|\hat u\|_{C^3}$ alone ---
  such that
  \begin{align}
    \label{eq:vLip}
    \|\velo(\T^1;\tm)-\velo(\T^2;\tm)\|_{L^2(\hat u)} \le L \|\T^1-\T^2\|_{C^2}.
  \end{align}
\end{lem}
\begin{proof}
  Introduce the linear interpolation $\T_s=s\T^1+(1-s)\T^2$.
  Clearly, $s\mapsto\T_s$ is a smooth curve in $C^2(\Omega)$, 
  and we have for all $s\in[0,1]$:
  \begin{align}
    \label{eq:interprop}
    \|\T_s\|_{C^2}\le \max\big(\|\T^1\|_{C^2},\|\T^2\|_{C^2}\big), 
    \quad
    \lambda\eins\le\dff\T_s\le\lambda^{-1}\eins.
  \end{align}
  Hence the curve $\Psi:[0,1]\to L^2(\hat u)$ with
  \begin{align*}
    \Psi(s)=\velo(\T_s;\tm) 
    = -\nabla\big[h'\big((\T_s)_\#\hat u\big)+V\big]\circ\T_s
    = -(\dff\T_s)^{-T}\nabla h'\left(\frac{\hat u}{\det\dff\T_s}\right) - \nabla V\circ\T_s
  \end{align*}
  is continuously differentiable, with derivative
  \begin{align*}
    \Psi'(s)
    &= (\dff\T_s)^{-T}(\dff\T^1-\dff\T^2)^T(\dff\T_s)^{-T}\nabla h'\left(\frac{\hat u}{\det\dff\T_s}\right) \\
    &\quad + (\dff\T_s)^{-T}\nabla\left(\left(\frac{\hat u}{\det\dff\T_s}\right)h''\left(\frac{\hat u}{\det\dff\T_s}\right)\tr\left[(\dff\T_s)^{-1}(\dff\T^1-\dff\T^2)\right]\right) \\
    &\quad + \nabla^2V\circ\T_s\cdot(\T^1-\T^2).
  \end{align*}
  Above, the $\T_s$ are differentiated at most twice in space.
  Using \eqref{eq:interprop} 
  and the fact that $r\mapsto h'(r)$ and $r\mapsto r h''(r)$ are smooth with bounded derivatives for $r\in[\lambda^{d+1},\lambda^{-(d+1)}]$, 
  it is easy to conclude that $\Psi'(s):\Omega\to\setR^d$ is actually a continuous function for each $s\in[0,1]$, 
  with
  \begin{align*}
    |\Psi'(s)| \le L\|\T^1-\T^2\|_{C^2}.
  \end{align*}
  Integrate this inequality with respect to $s\in[0,1]$ to obtain \eqref{eq:vLip}.
\end{proof}
\begin{lem}
  \label{lem:convex}
  Assume that $\T^1,\T^2\in\nabla\xspc_\theh^+$ satisfy the set of estimates
  \begin{align}
    \label{eq:mono1}
    \left|\left\langle\frac{\T^j-\id}\tau - \velo(\T^j;\tm),\nabla\psi\right\rangle_{L^2(\tm_\#\rfd)}\right|
    \le \rho_j\|\nabla\psi\|_{L^2(\tm_\#\rfd)}
  \end{align}
  for all $\psi\in\txspc_\theh$, with constants $\rho_1,\rho_2\ge0$.
  Then
  \begin{align}
    \label{eq:resest}
    \|\T^1-\T^2\|_{L^2(\tm_\#\rfd)} \le \frac\tau{1+\tau\lambda}(\rho_1+\rho_2).
  \end{align}
\end{lem}
\begin{proof}
  Subtracting the respective inequalities \eqref{eq:mono1} for the same $\psi\in\txspc_\theh$ from each other 
  yields
  \begin{align*}
    \left|\left\langle\frac{\T^1-\T^2}\tau 
    - \left(\velo(\T^1;\tm)-\velo(\T^2;\tm)\right),\nabla\psi\right\rangle_{L^2(\tm_\#\rfd)}\right|
    \le (\rho^1+\rho^2)\|\nabla\psi\|_{L^2(\tm_\#\rfd)}.
  \end{align*}
  We choose $\psi\in\xspc_\theh$ such that $\T^1-\T^2=\nabla\psi$;
  with Corollary \ref{cor:lambda} of $\nrjx$'s $\lambda$-convexity,
  we arrive at
  \begin{align*}
    (\rho^1+\rho^2)\|\T^1-\T^2\|_{L^2(\tm_\#\rfd)}
    &\ge \left|\frac1\tau\|\T^1-\T^2\|^2_{L^2(\tm_\#\rfd)} 
      - \left\langle \T^1-\T^2,\velo(\T^1;\tm)-\velo(\T^2;\tm)\right\rangle_{L^2(\tm_\#\rfd)}\right| \\
    &\ge \left(\frac1\tau+\lambda\right) \|\T^1-\T^2\|^2_{L^2(\tm_\#\rfd)}.
  \end{align*}
  Division by the norm of $\T^1-\T^2$ yields \eqref{eq:resest}.
\end{proof}

\subsection{Weak convergence}
\label{sct:consist}
\begin{proof}[Proof of Theorem \ref{thm:consist}]
  We shall prove the following estimate,
  \begin{align}
    \label{eq:induct}
    \wass(u_{n\tau},u_\Delta^n) \le \big(1+\bigO(\tau)\big)\wass(u_{(n-1)\tau},u_\Delta^{n-1}) + \tau\bigO(\tau+\rho_\theh).
  \end{align}
  The result \eqref{eq:theresult} then trivially follows by induction on $n$.

  Hence fix some $n\in\setN$.
  Recall that $u_\Delta^n=(\T_\Delta^n)_\#\hat u$,
  where he have introduced $\hat u:=u_\Delta^{n-1}$ for brevity. 
  Define further the flow maps $\tilde\T_{(\cdot)}$ as in \eqref{eq:tildeflow},
  and the density $u_\tau^*=\proj_\theh[\tilde\T_\tau]_\#\hat u$.
  Now we apply the triangle inequality as follows:
  \begin{align}
    \label{eq:threeterms}
    \wass(u_{n\tau},u_\Delta^n) \le \wass(u_{n\tau},\tilde u_\tau)+\wass(\tilde u_\tau,\tilde u_\tau^*)+\wass(\tilde u_\tau^*,u_\Delta^n),
  \end{align}
  and we shall estimate the three terms on the right-hand side separately.
  The first one is easy:
  since \eqref{eq:eq} defines a $\lambda$-contractive gradient flow in $\wass$,
  \begin{align*}
    \wass(u_{n\tau},\tilde u_\tau) \le e^{-\lambda\tau}\wass(u_{(n-1)\tau},u_\Delta^{n-1}).
  \end{align*}
  To estimate the second term on the right-hand side of \eqref{eq:threeterms}, 
  recall that $\tilde u_\tau=(\tilde\T_\tau)_\#\hat u$ and $\tilde u_\tau^*=\proj_\theh[\tilde\T_\tau]_\#\hat u$.
  We thus have, with the help of \eqref{eq:Ttau2}:
  \begin{align*}
    \wass(\tilde u_\tau,\tilde u_\tau^*)
    \le \|\tilde\T_\tau-\proj_\theh[\tilde\T_\tau]\|_{L^2(\hat u)}
    \le \|\tilde\T_\tau-\proj_\theh[\tilde\T_\tau]\|_{C^0}
    \le \tau\left\|\frac{\tilde{\T}_\tau-\proj_\theh[\tilde{\T}_\tau]}\tau\right\|_{C^2} \le C\tau(\tau+\rho_\theh).
  \end{align*}
  Finally, to control the third distance in \eqref{eq:threeterms}, 
  we use that $\tilde u_\tau^*=\proj_\theh[\tilde\T_\tau]_\#\hat u$ and that $u_\Delta^n=(\T_\Delta^n)_\#\hat u$,
  so that
  \begin{align*}
    \wass(\tilde u_\tau^*,u_\Delta^n) \le \|\proj_\theh[\tilde\T_\tau]-\T_\Delta^n\|_{L^2(\hat u)}.
  \end{align*}
  Here we wish to apply Lemma \ref{lem:convex} above, with $\T^1=\proj_\theh[\tilde{\T}_\tau]$ and $\T^2=\T_\Delta^n$.
  Clearly, $\rho_2=0$.
  To estimate $\rho_1$, we provide a bound on each of the three terms on the right-hand side of
  \begin{align*}
    \left\|\frac{\proj_\theh[\tilde{\T}_\tau]-\id}\tau-\velo(\proj_\theh[\tilde{\T}_\tau];\tm_\Delta^{n-1})\right\|_{L^2(\hat u)}
    &\le \begin{cases}
      \frac1\tau\left\|\proj_\theh[\tilde{\T}_\tau]-\tilde{\T}_\tau\right\|_{L^2(\hat u)} \\
      + \left\|\frac{\tilde{\T}_\tau-\id}\tau-\velo(\tilde{\T}_\tau;\tm_\Delta^{n-1})\right\|_{L^2(\hat u)} \\
      +\left\|\velo(\tilde{\T}_\tau;\tm_\tau^{n-1})-\velo(\proj_\theh[\tilde{\T}_\tau];\tm_\Delta^{n-1})\right\|_{L^2(\hat u)}.
    \end{cases}
  \end{align*}
  For the first term above, we have, again by \eqref{eq:Ttau2}, that
  \begin{align*}
    \frac1\tau\left\|\proj_\theh[\tilde{\T}_\tau]-\tilde{\T}_\tau\right\|_{L^2(\hat u)} \le C(\tau+\rho_\theh).
  \end{align*}
  To estimate the second term, recall that the relation \eqref{eq:helpy} holds,
  and therefore \eqref{eq:Ttau} yields
  \begin{align*}
    \left\|\frac{\tilde{\T}_\tau-\id}\tau-\velo(\tilde{\T}_\tau;\tm_\Delta^{n-1})\right\|_{L^2(\hat u)}
    \le \|\tau\zeta\|_{L^2(\hat u)} \le \tau\|\zeta\|_{C^0} \le \tau\|\zeta\|_{C^3} \le B'\tau.
  \end{align*}
  In the final step, we use the Lipschitz continuity \eqref{eq:vLip} of the map $\velo(\;\cdot\;;\tm_\Delta^{n-1})$ from $C^2$ to $L^2(\hat u)$,
  and apply \eqref{eq:Ttau2} again:
  \begin{align*}
    \left\|\velo(\tilde{\T}_\tau;\tm_\Delta^{n-1})-\velo(\proj_\theh[\tilde{\T}_\tau];\tm_\Delta^{n-1})\right\|_{L^2(\hat u)}
    \le L \|\tilde{\T}_\tau-\proj_\theh[\tilde{\T}_\tau]\|_{C^2}
    \le LC\tau(\tau+\rho_\theh).
 \end{align*}
 This concludes the proof of \eqref{eq:induct}, 
 from which the final result \eqref{eq:theresult} can be easily deduced.
\end{proof}

%
\section{Implementation}
\label{sct:implement}
%
In practice, performing one time step in the iterative scheme means to find the unique minimizer $\T^n_\Delta\in\nabla\xspc^+_\theh$ 
of the convex functional
\begin{align}
  \label{eq:num-el}
  \T \mapsto 
  \frac1{2\tau}\intom \big|\T\circ\tm^{n-1}_\Delta-\tm^{n-1}_\Delta\big|^2\,\rfd\dd x
  + \intom h_*\left(\frac{\det\dff\tm_\Delta^{n-1}}{\rfd}\,\big(\det\dff\T\big)\circ\tm_\Delta^{n-1}\right)+V(\T\circ\tm_\Delta^{n-1})\,\rfd\dd x.
\end{align}
Although this minimization problem is finite-dimensional,
its numerical solution is not immediate, 
since the integrals cannot be evaluated explicitly.
Even if $\rfd$ and $h_*$ happen to be ``nice'', the main difficulty remains, 
namely the appearance of $\tm_\Delta^{n-1}$, 
which is an $(n-1)$-fold concatenation of transport maps $\T_\Delta^m$ from $\nabla\xspc^+_\theh$. 

Therefore, a further simplification of the minimization problem is needed.
Our approach is to replace the integrals in \eqref{eq:num-el} by suitable quadrature rules.
That is, the integrands are evaluated only at a fixed number $K$ of quadrature points $\check x_k\in\Omega$,
and these values are then summed against given weights $\check\omega_k>0$.  These are chosen as follows: Fix two positive integers $K_1$ and $K_2$.
Decompose $\Omega=[0,1]^d$ into $K_1^d$-many cubes of equal size.
It is sensible to choose $K_1$ proportional to the number $\ell$ of modes per direction.
On each of these cubes, introduce $K_2^d$-many quadrature points using a one-dimensional quadrature rule on a tensor product grid.

The ``complicated'' function $\tm_\Delta^{n-1}$ appears in two ways when using quadrature on \eqref{eq:num-el}:
first, in the positions $\check x^{n-1}_{\Delta,k}:=\tm_\Delta^{n-1}(\check x_k)$,
and second, in the determinants $\check\sigma^{n-1}_{\Delta,k}:=\det\dff\tm_\Delta^{n-1}(\check x_k)$.
Fortunately, these quantities can easily be determined by iteration, using that $\tm_\Delta^n=\T_\Delta^n\circ\tm_\Delta^{n-1}$:
\begin{align*}
  \check x^0_{\Delta,k}=\check x_k,\quad &\check x^n_{\Delta,k}=\T_\Delta^{n}(\check x^{n-1}_{\Delta,k}), \\
  \check\sigma^0_{\Delta,k}=1, \quad &\check\sigma^n_{\Delta,k}=\det\dff\T_\Delta^{n}(\check x^{n-1}_{\Delta,k})\,\check\sigma^{n-1}_{\Delta,k}.
\end{align*}
Notice that this iteration is explicit in the sense that the quantities $\check x^n_{\Delta,k}$ and $\check\sigma^n_{\Delta,k}$
can be computed (in order to determine $\T^{n+1}_{\Delta,k}$ in the next time step)  \emph{after} the minimizer $\T^n_{\Delta,k}$ in \eqref{eq:num-el} has been obtained.
The functional in \eqref{eq:num-el} now attains the form
\begin{align}
  \label{eq:num-el2}
  \T\mapsto
  \frac1{2\tau} \sum_k\big|\T(\check x^{n-1}_{\Delta,k})-\check x^{n-1}_{\Delta,k}\big|^2\,\rfd(\check x_k)\check\omega_k
  + \sum_k h_*\left(\frac{\check\sigma^{n-1}_{\Delta,k}}{\rfd(\check x_k)}\,\det\dff\T(\check x^{n-1}_{\Delta,k})\right)+V(\T(\check x^{n-1}_{\Delta,k}))\;\rfd(\check x_k)\check\omega_k,
\end{align}
with the quantities $\check x^{n-1}_{\Delta,k}$ and $\check\sigma^{n-1}_{\Delta,k}$, 
that have been calculated in the previous time step.
Observe that in \eqref{eq:num-el2}, 
the reference densities $\rfd(\check x_k)$ are constants, 
and that $\T(\check x^{n-1}_{\Delta,k})$ and $\det\dff\T(\check x^{n-1}_{\Delta,k})$ are easy to compute for any given $\T\in\nabla\xspc^+_\theh$.
\begin{lem}
  The minimization problem \eqref{eq:num-el2} is strictly convex in $\T\in\nabla\xspc_\theh$.
  There exists a minimizer, which is unique if it belongs to $\nabla\xspc^+_\theh$.
\end{lem}
\begin{rmk}
  In contrast to Theorem \ref{thm:well-posed},
  it is not a conclusion that the minimizer $\T^*$ belongs to $\nabla\xspc^+_\theh$.
  The problem is that $\dff\T^*$ might in principle degenerate at any point $x\in\Omega$
  except for the $\check x_k$.
\end{rmk}
\begin{proof}
  Convexity follows along the same lines as in the proof of Proposition \ref{prp:convex}.
  Indeed, the core of the argument is the verification of the \emph{pointwise} convexity of the integrand in the definition $\nrjx$.
  In \eqref{eq:num-el2}, the integral is replaced by a summation.
  To obtain the uniqueness, it suffices to observe that $\xspc^+_\theh$ is in the interior of the convex set $\xspc^+_\theh$.
\end{proof}

\section{Numerical results}
\label{sct:numerics}
%
We present the results of a numerical experiment in space dimension $d=2$ and for a quadratic pressure function $\prss(s) = s^2$.
Even though this is a borderline case for which Theorem \ref{thm:well-posed} does \emph{not} guarantee 
the well-posedness of the time-discrete iteration, our scheme works well and produces positive bounded fully discrete solutions.
For the quadrature rule in \eqref{eq:num-el2}, we choose $K_1=2K$ and $K_2=2$. 
Thus, $\Omega=[0,1]^2$ is decomposed into $4K^2$ squares, 
with four quadrature points $\check x_k$ per square.
The nodes $\check x_k$ and the weights $\check\omega_k$ are chosen according to the Gau\ss{} quadrature rule.

\subsection{Reference solution}
In order to numerically estimate the order of convergence of our scheme,
we use a reference solution $u_\rf$ that is produced by a highly resolved FEM method with fully implicit time stepping.
As ansatz functions $\varphi_k$, we choose tensor products of hat functions with respect to a standard square lattice. 
The $n$th time iterate $u_\rf^n$ of the reference solution is then obtained by solving
\begin{align}
  \label{eq:FEM}
  \intom\frac{u - u_\rf^{n-1}}{\tau}\,\varphi_k \dd x
  = \intom \big[\nabla(u^2)+ u\nabla V\big]\cdot\nabla\varphi_k \dd x 
  \quad \text{for all $k$}
\end{align}
for $u = \sum_k u^n_k\varphi_k$.
We use $400$ lattice points in both spatial directions, and a time step of $\tau=5\cdot10^{-4}$.

\subsection{Numerical experiments}
\subsubsection{Qualitative behavior}
\label{exp1}
In the first series of experiments, 
we choose the potential $V$ as
\begin{align}
  \label{eq:V_exp1}
  V(x) = -\lambda\big(\cos(2\pi x_1)-1\big)\big(\cos(4\pi x_2)-1\big),
\end{align}
with $\lambda=0.75$.
We approximate solutions for the initial density
\begin{align}
  \label{eq:u0_exp1}
  u^0 = C\left(0.1 + x_1\big(\cos(4\pi x_1) - 1.2\big)\big(\cos(2\pi y_2)-1\big)\right),
\end{align}
where $C>0$ is such that $u^0$ has unit mass.

We employ our scheme with a relatively high spatial resolutions of $K=32$.
We visualize the result in Figure \ref{fig:fig1_2D}:
the left column displays the transport maps $\tm_\Delta^n$, 
and the right column shows the corresponding density functions $u_\Delta^n$ 
after various numbers $n$ of time iterations.
Qualitatively, the density evolves from the initial state that has two peaks of different height
to the an approximation of the stationary solution that has two peaks of equal height,
which are oriented in the orthogonal direction.
\begin{figure}
\begin{center}
  \gridplot{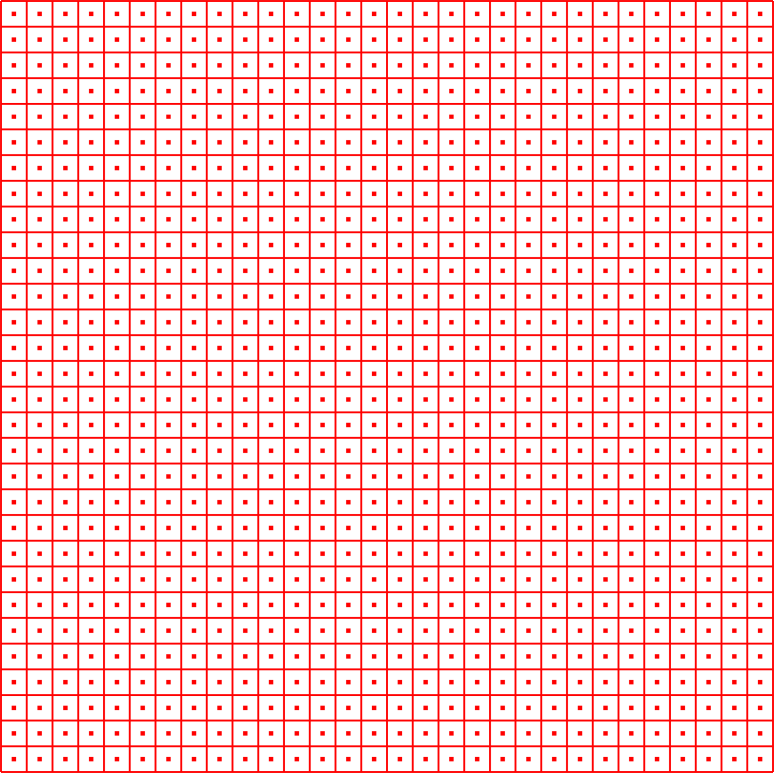}
  \gridplot{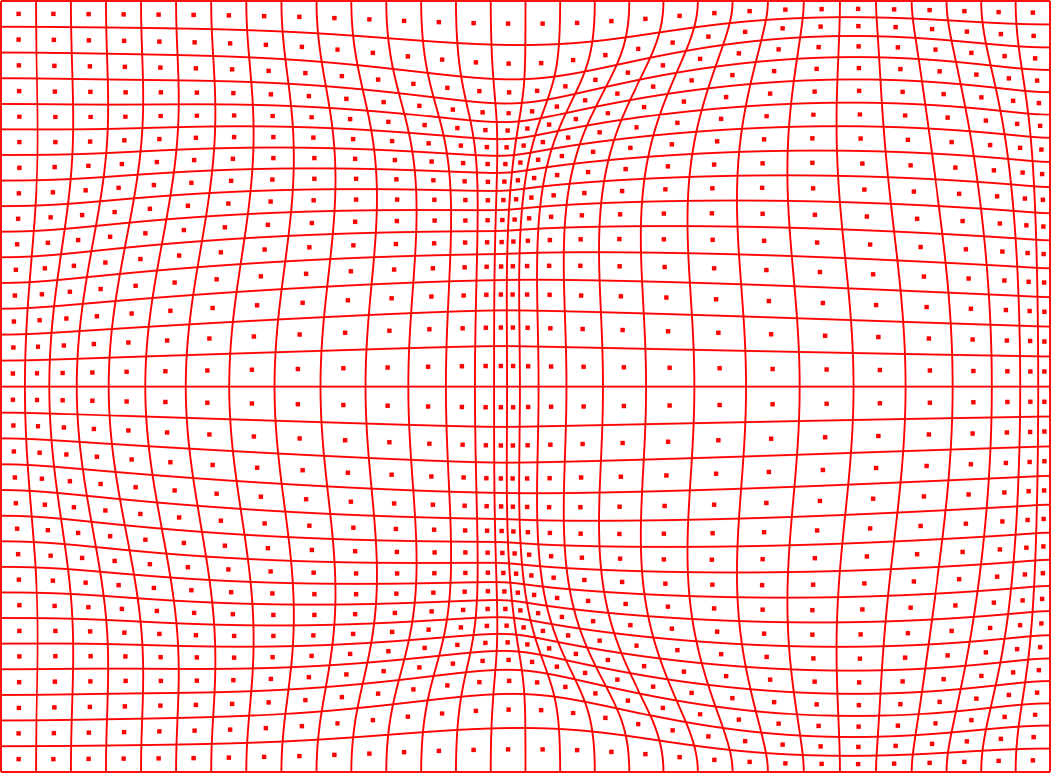}
  \gridplot{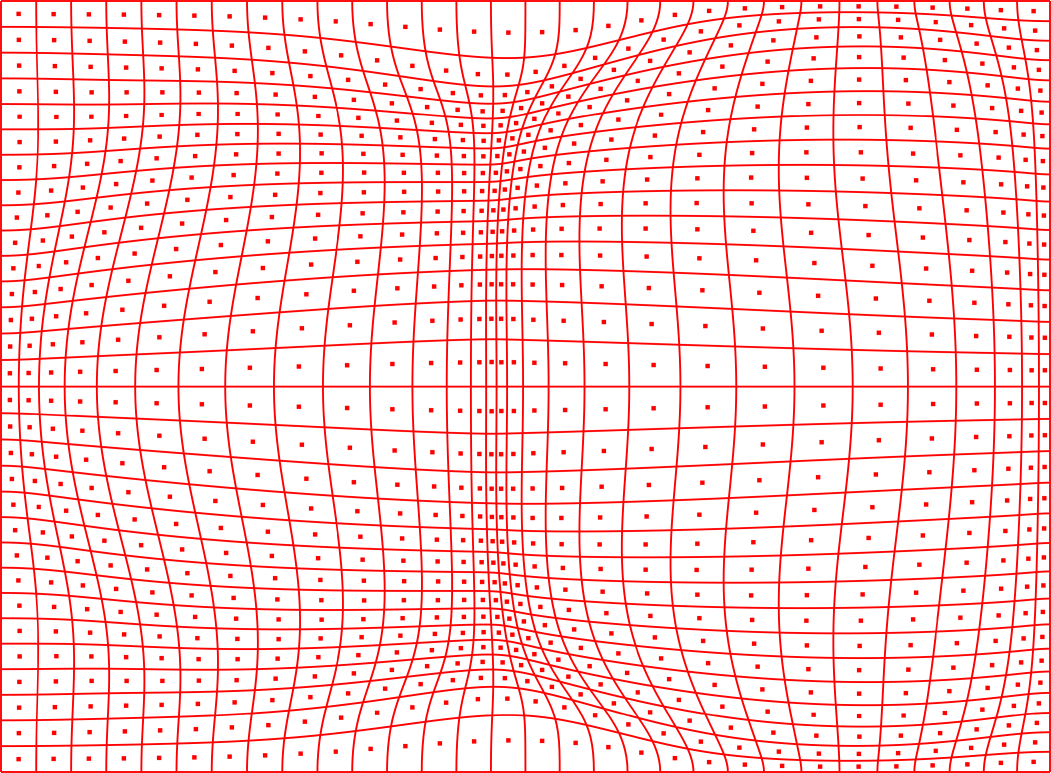}
  \gridplot{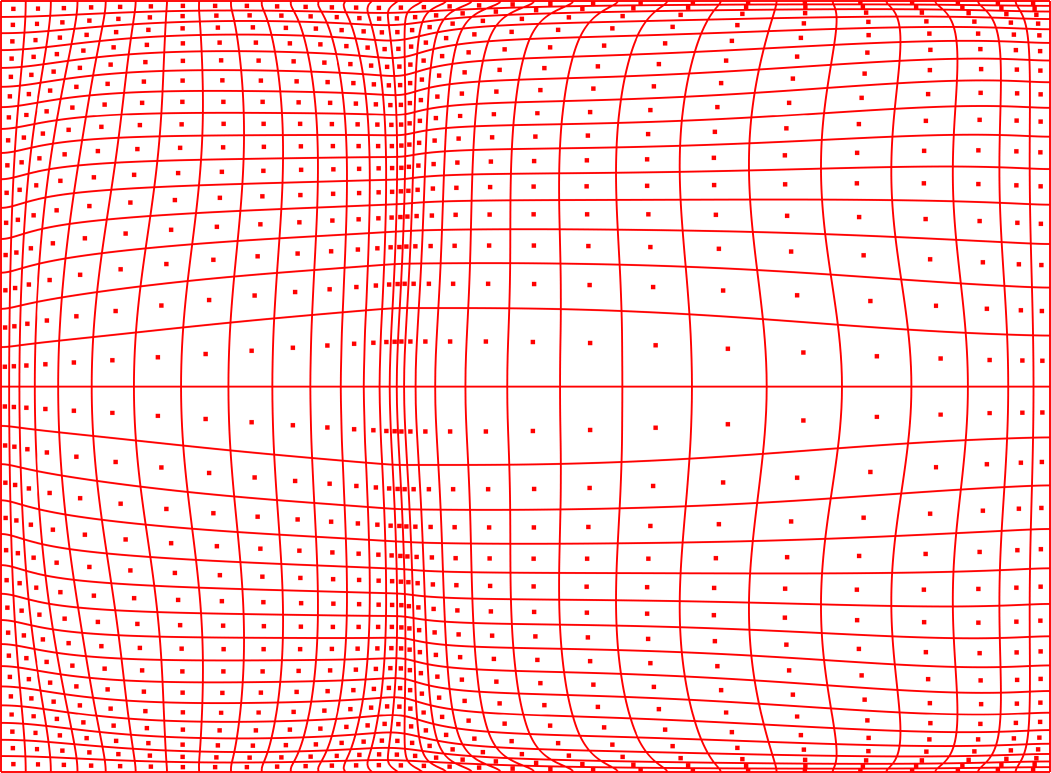}\\
  \densplot{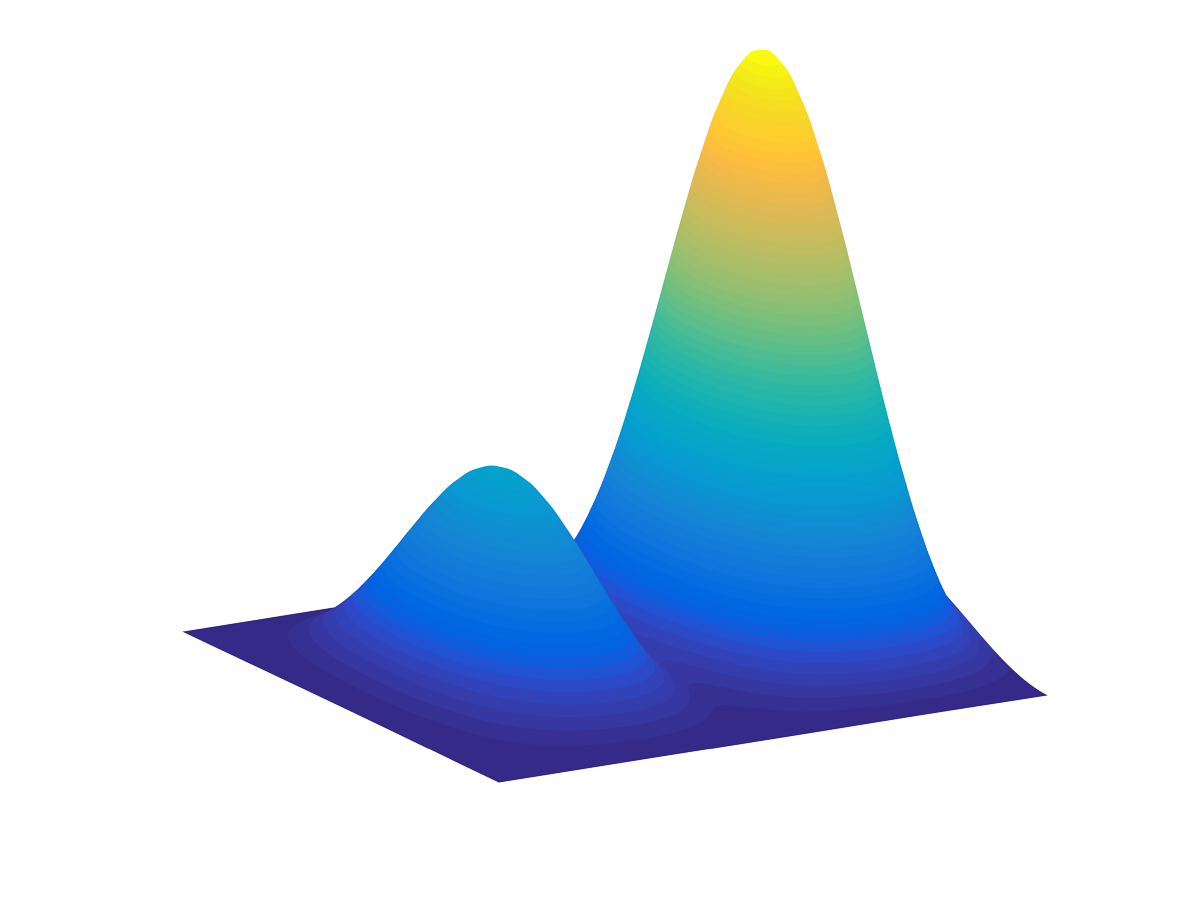}
  \densplot{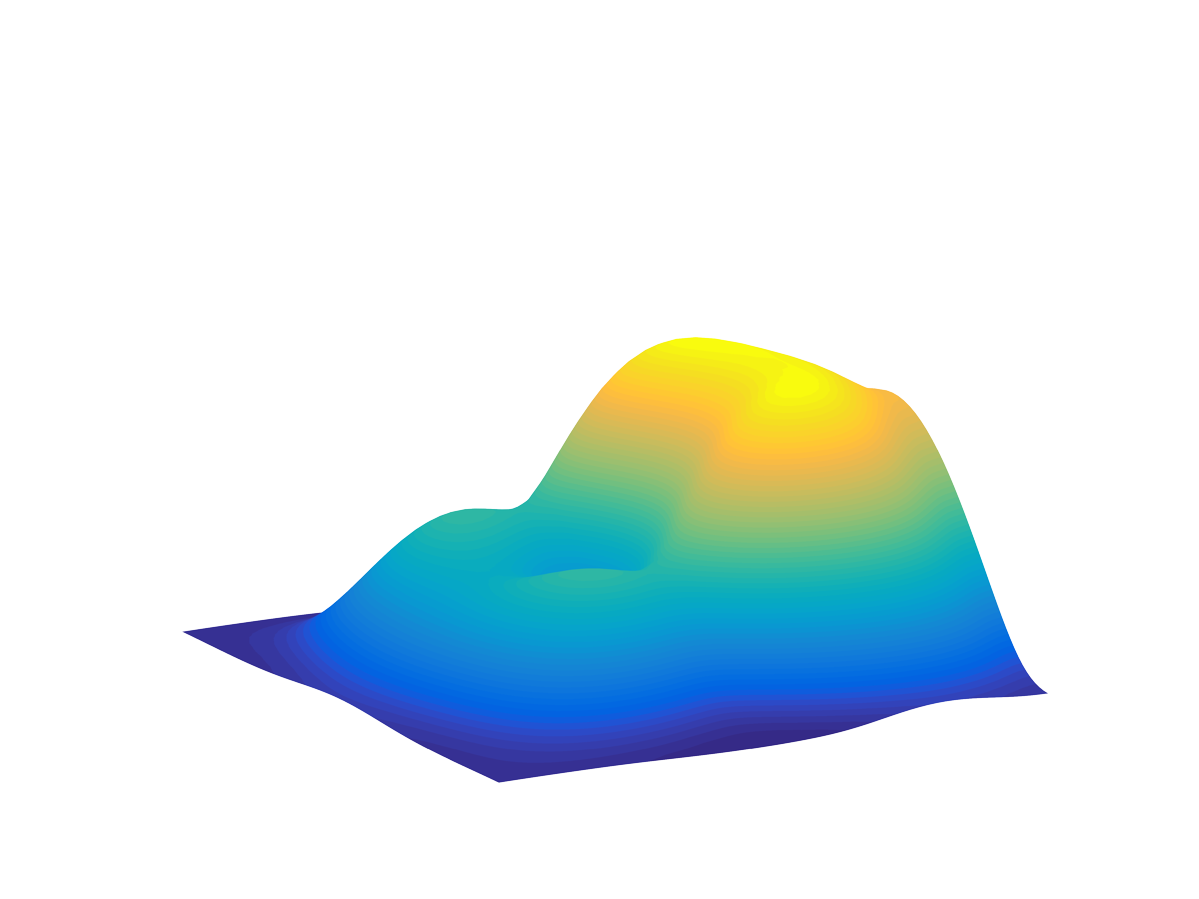}
  \densplot{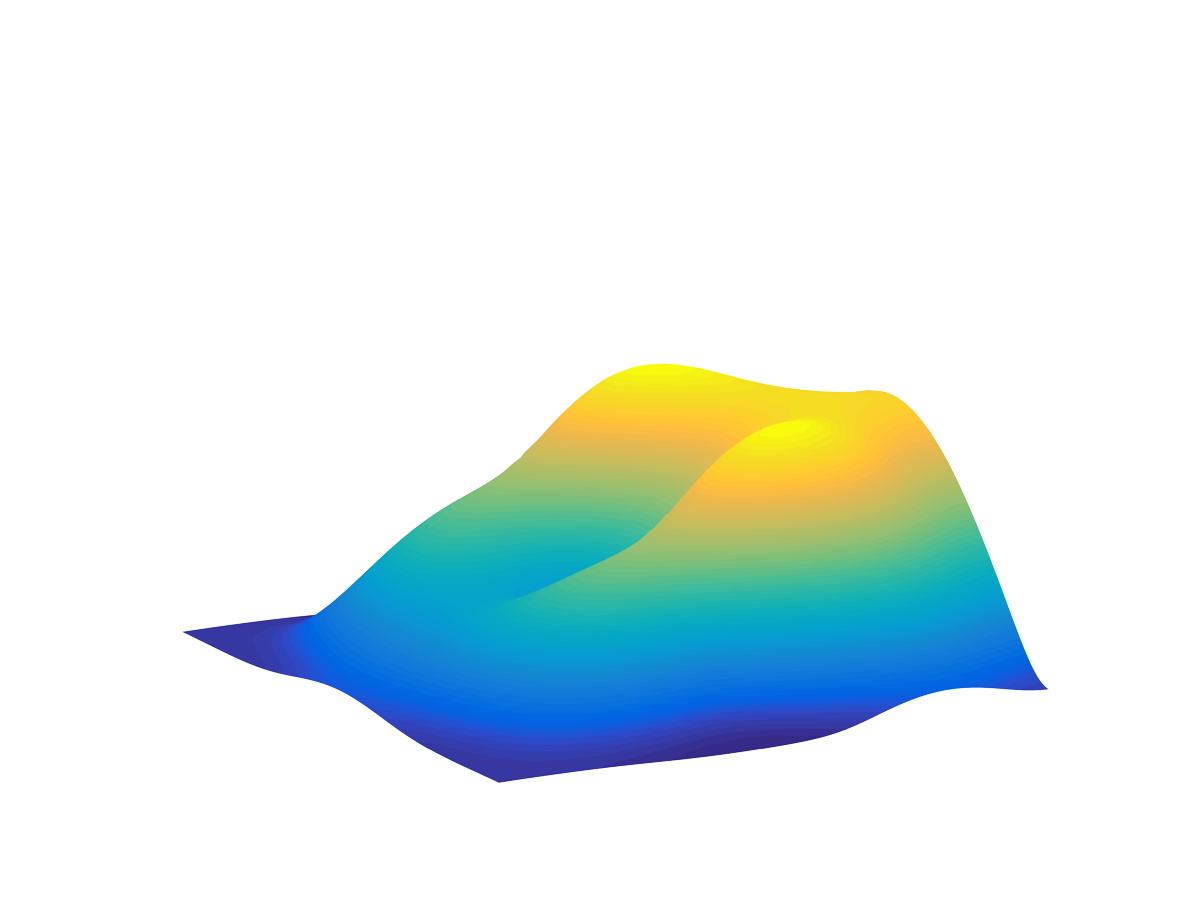}
  \densplot{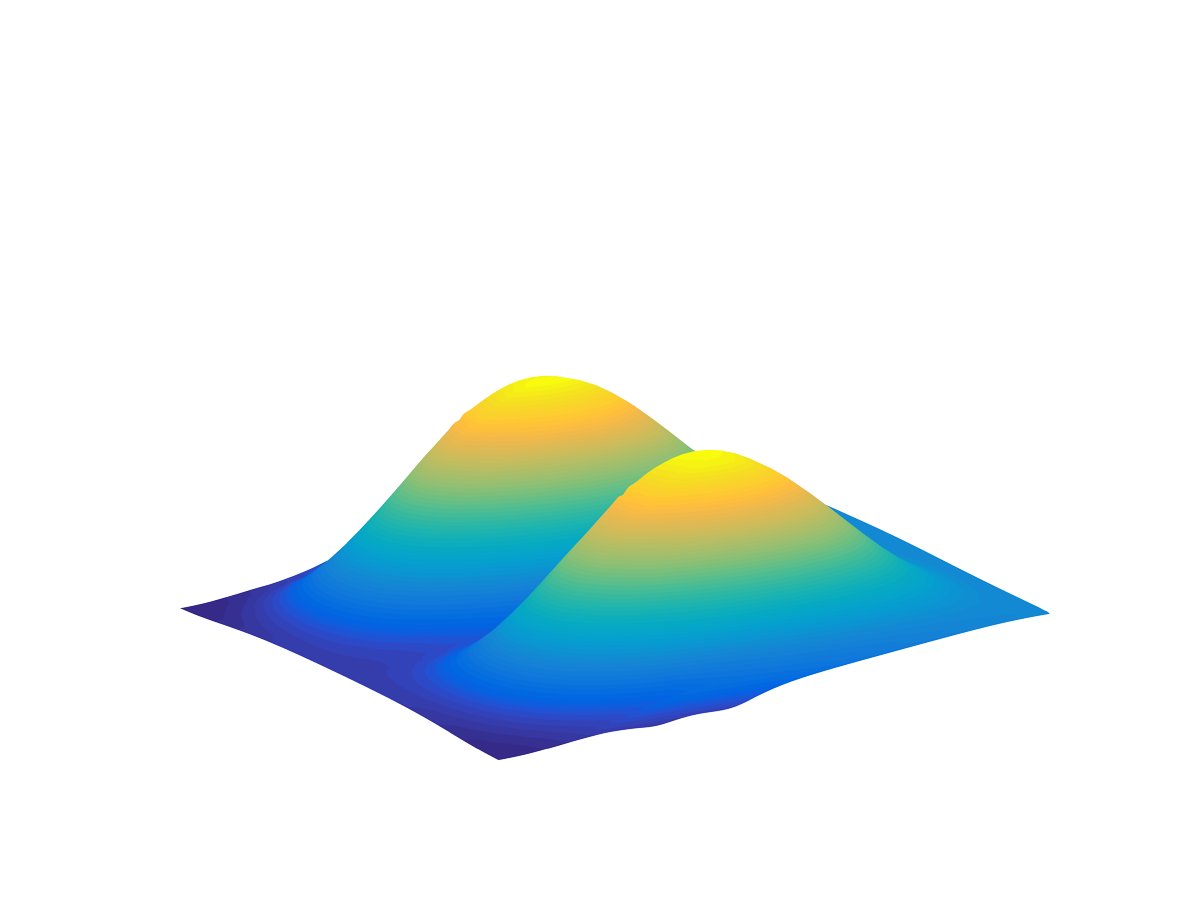}
\end{center}
\caption{Snapshots of the evolution of the fully discrete solution for the initial density \eqref{eq:u0_exp1},
  taken at times $t=0,2.5\cdot 10^{-3},4\cdot10^{-3}$ and $t=5\cdot10^{-2}$.
  \emph{Upper row:} transport maps $\tm_\Delta^n$;
  \emph{Lower row:} densities $u_\Delta^n$.}
\label{fig:fig1_2D}
\end{figure}
Figure \ref{fig:fig2_2D}/middle shows the correponding decay of the entropy $\nrjx$.
This curve is compared to the one obtained for $\anrj(u_\rf)$, 
using the reference solution with the same initial datum $u^0$.
The curves are virtually indistinguishable.

\subsubsection{Rate of convergence}
\label{exp3}
For a quantitative estimate on the accuracy of our numerical scheme,
we eliminate the external potential, $V\equiv0$,
and run our scheme with the same initial datum $u^0$ as above in \eqref{eq:u0_exp1},
with various different spatial resolutions, ranging from $\ell=8$ (coarsest) to $\ell=24$ (finest).
The time step width $\tau=5\cdot 10^{-4}$ is kept fixed.

These solutions are then compared to the reference solution $u_\rf$ 
that has been obtained by the finite element scheme \eqref{eq:FEM}.
In Figure \ref{fig:fig2_2D}/left, we have plotted the behavior of the $L^2$-norm of the difference,
evaluated at the time $T=0.01$.
The observed order of convergence is between two and three, i.e.,
the approximation error is roughly inverse proportional to the number $K^2$ of modes.
\begin{figure}[t]
  \hspace*{-1.5cm}
  \begin{tikzpicture}
	\tikzstyle{every node}=[font=\footnotesize]
    \begin{loglogaxis}[
      enlargelimits=false,
      width=0.39\textwidth,
      xlabel={number $K$ of modes},
      title={$\|u_\Delta(T)-u_\rf(T)\|_{L^2}$},
      xtick={4,8,12,16,20,24},
      xticklabel=\pgfmathparse{exp(\tick)}\pgfmathprintnumber{\pgfmathresult},
      mark size=1.5pt,
      grid,
      legend entries={$L^2$-error,$\mathcal{O}(K^{-2})$,$\mathcal{O}(K^{-3})$},
      legend style={at={(0.275,0)},anchor=south},
      ]
      \addplot[blue,mark=*] table {error_1.dat};
      \addplot[black,mark=+,domain=4:24,samples=20] {10*x^(-2)};
      \addplot[black,mark=o,domain=4:24,samples=20] {30*x^(-3)};
    \end{loglogaxis}
  \end{tikzpicture}\hspace*{-4mm}
  \begin{tikzpicture}
	\tikzstyle{every node}=[font=\footnotesize]
    \begin{axis}[
      enlargelimits=false,
      width=0.39\textwidth,
      xlabel={time $t$},
      xtick={0,0.02,0.04},
      mark size=1.5pt,
      grid,
      legend entries={GF,FEM},
      title = entropy
      ]
      \addplot[red,mark=*] table {decay_1.dat};
      \addplot[black,mark=o] table {decay_2.dat};
    \end{axis}
  \end{tikzpicture} 
    \begin{tikzpicture}
	\tikzstyle{every node}=[font=\footnotesize]
    \begin{semilogyaxis}[
      enlargelimits=false,
      width=0.39\textwidth,
      xlabel={time $t\sim n\tau$},
      xtick={0,0.025,0.05,0.075,0.1},
      mark size=1.5pt,
      grid,
      legend entries = {$\|\tm_{\Delta}^0 - \tm_{\Delta}^0\|_{L^2(u^0)}\cdot e^{-\lambda n\tau}$,
      $\|\tm_{\Delta}^n - \tm_{\Delta}^n\|_{L^2(u^0)}$},
      legend style={at={(0.5,0.05)},anchor=south},
      ymin=1e-4,
      title=error
      ]
      \addplot[red,thick,dashed] table {convex_2.dat};
      \addplot[black,thick,mark=o] table {convex_1.dat};
    \end{semilogyaxis}
  \end{tikzpicture}

  \caption{\emph{Left:} $L^2$-distance between the solution of our variational scheme and the reference solution  at the final time $T=0.01$, using $K=4,8,12,16,20,25$ modes.
    \emph{Middle:} Decay of the entropy. \emph{Right:} Exponential rate of decay in time of the $L^2$-distance between two particular fully discrete solutions.}
  \label{fig:fig2_2D}
\end{figure}
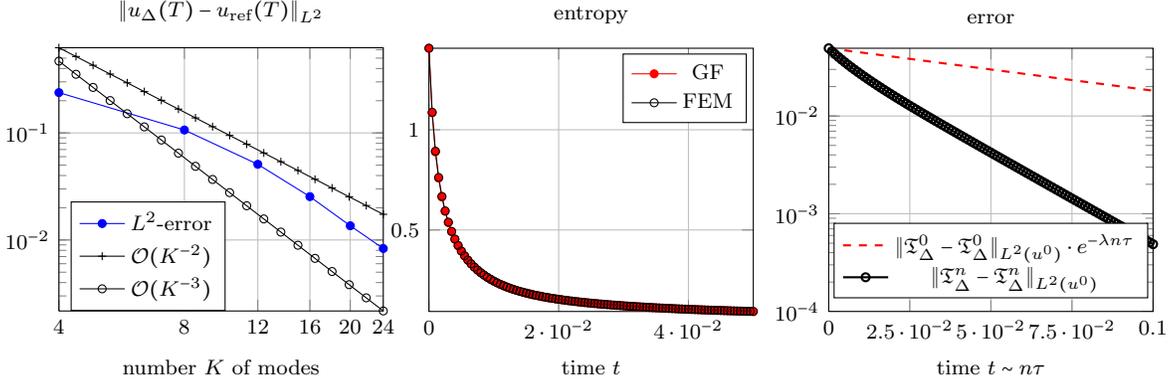

\subsubsection{Contractivity}\label{exp4}
Gradient flows of $\lambda$-convex functionals are $\lambda$-contractive.
In the context at hand, this means that if the potential $V$ in \eqref{eq:eq} is $\lambda$-convex,
and $u_1$, $u_2$ are two solutions to the initial value problem \eqref{eq:eq}--\eqref{eq:ic} 
with respective initial data $u_1^0$, $u_2^0$, 
then
\begin{align}
  \label{eq:wass_contraction}
  \wass(u_1(t),u_2(t)) \le \wass(u_1^0,u_2^0) e^{-\lambda t}.
\end{align}
Since our discretization preserves $\lambda$-convexity as detailed in Propostion \ref{prp:convex},
a corresponding contraction estimate holds --- in each time step, see Corollary \ref{cor:lambda}.
However, since the $L^2$-space on which the minimization is performed changes from one time step to the next, 
it is not clear how to iterate that estimate and obtain an analogue of \eqref{eq:wass_contraction} above.

A natural question to ask is if maybe a stronger contraction principle holds, 
namely if the $L^2$-distance (with respect to $\rfd$) between the transport maps $\tm_\Delta^n$ 
for two solutions satisfies the analogous estimate as \eqref{eq:wass_contraction} for the Wasserstein distance between densities.
In general, such a stronger principle is too much to ask for (recall that $\anrj$ is just poly-convex).
However, some numerical experiments suggest that this estimate does hold if the initial data are already close in a suitable sense.

For illustration, 
we have performed a numerical experiment for $V(x) = \frac{\lambda}{2}\|x\|_2^2$ with $\lambda=10$,
with two initial conditions $u_1^0$ and $u_2^0$ that are random perturbations of the unit density $\rfd=1$. 
More precisely, $u_j^0=(\tm_j^0)_\#\rfd$ for $j=1,2$, 
with $\tm_1^0,\tm_2^0\in\xspc_\theh$ being random perturbations of the identity.
The discretization parameters are $K=12$ and $\tau=10^{-3}$.
The result is given in Figure \ref{fig:fig2_2D}/right;
there is apparently an exponential contraction of the transport maps, 
with an exponential rate that is higher (but still comparable to) $\lambda$.

\subsubsection{Code}

In principle, an actual implementation of our scheme can be quite compact as the  conceptual code for the Julia language in Fig.\ \ref{fig:code} shows.  Note, however, that here we trade elegance for speed and that for the experiments reported on above we have actually used a much faster implementation. 

\begin{figure}
\begin{lstlisting}
using ForwardDiff
d = 2; K = 5                                                    # dimension, number of modes
g(k,x) = cos(pi*k*x)			
dg(k,x) = -pi*k*sin(pi*k*x)
ddg(k,x) = -pi^2*k^2*cos(pi*k*x)
gradb(k,x) = 2/(pi*norm(k))*[dg(k[1],x[1])*g(k[2],x[2]); g(k[1],x[1])*dg(k[2],x[2])]
Hb(k,x) = 2/(pi*norm(k))*[ddg(k[1],x[1])*g(k[2],x[2]) dg(k[1],x[1])*dg(k[2],x[2]);
                             dg(k[1],x[1])*dg(k[2],x[2]) g(k[1],x[1])*ddg(k[2],x[2])];
I = [[k1,k2] for k1=0:K, k2=0:K][2:end]
t(z,x) = x + sum([z[k[2]*(K+1)+k[1]]*gradb(k,x) for k = I])     # Lagrangian maps
Dt(z,x) = eye(2) + sum([z[k[2]*(K+1)+k[1]]*Hb(k,x) for k = I])
hs(r) = 1/r                                                     # pressure function
L = 20; dx = 1/(L-1); dy = dx; 									
X = [[x,y] for x = 0:dx:1, y = 0:dy:1]                          # space grid
om1 = ones(L); om1[1] = om1[L] = 1/2; 
om = [ a*b*dx^2  for a = om1, b = om1];                         # weights for quadrature
u0(x) = (0.1+x[1]*(cos(4*pi*x[1])-1.2)*(cos(2*pi*x[2])-1))
C = sum([u0(x)*dx^2 for x = X])
ub = [u0(X[k,l])/C for k=1:L, l=1:L]                            # reference density
V(x) = -0.75*(cos(2*pi*x[1])-1)*(cos(4*pi*x[2])-1)              # potential
si = ones(L,L)                                                  # determinants
tau = 5e-4                                                      # time step size
int(f) = sum(f.*ub.*om)
W(z) = int([sum((t(z,X[k,l])-X[k,l]).^2) for k=1:L, l=1:L])
E(z) = int([(hs(si[k,l]/ub[k,l]*det(Dt(z,X[k,l]))) + V(t(z,X[k,l]))) for k=1:L, l=1:L])
F(z) = 1/(2*tau)*W(z) + E(z) 
gradF = ForwardDiff.gradient(F)
HF = ForwardDiff.hessian(F)
z = zeros(size(I));
time = 0
while time < 5*1e-2                                             # time stepping
    while norm(gradF(z)) > 1e-3                                 # Newton iteration
        dz = -HF(z)\gradF(z)
       z = z + dz
    end
    si = [det(Dt(z,X[k,l]))*si[k,l] for k = 1:L, l = 1:L]       # updating determinants
    X = [t(z,X[k,l]) for k = 1:L, l = 1:L];                     # updating space grid
    time = time + tau
end
\end{lstlisting}
\caption{\label{fig:code}Conceptual code of the Lagrangian scheme in the Julia language.}
\end{figure}

%


\appendix

\section{Gradient flows}
\label{apx:flows}
\begin{lem}
  \label{lem:2equivalent}
  Let two functionals $\anrj:\dens\to\setR$ and $\nrjx:\maps\to\setR$ be given, 
  which satisfy $\nrjx(\tm)=\anrj(\tm_\#\rfd)$, for some reference density $\rfd$.
  Then the derivatives of these functionals at any given $\tm\in\maps$,
  are related as follows:
  \begin{align}
    \label{eq:W2L2}
    \nabla\isom{L^2(\dn x)}{\dff\anrj(\tm_\#\rfd)}\circ\tm = \isom{L^2(\rfd)}{\dff\nrjx(\tm)}.
  \end{align}
\end{lem}
\begin{proof}
  Let $\zeta\in C^\infty_c(\Omega;\setR^d)$ be a smooth vector field.  
    Given $\tm$, consider the family $(\tm^s)_{s\in\setR}$ of diffeomorphisms $\tm^s$ of $\Omega$
  that is defined as solution to the ODE problem
  \begin{align*}
    \partial_s\tm^s = \zeta\circ \tm^s, \quad \tm^0=\tm.
  \end{align*}
  Be definition of $\nrjx$ via $\nrjx(\tm)=\anrj(\tm_\#\rfd)$,
  we trivially have that
  \begin{align}
    \label{eq:triveq}
    \frac{\dn}{\dn s}\bigg|_{s=0}\nrjx(\tm^s) = 
    \frac{\dn}{\dn s}\bigg|_{s=0}\anrj(\tm^s_\#\rfd). 
  \end{align}
  We shall now compute both derivatives separately.
  On the one hand,
  \begin{align*}
    \frac{\dn}{\dn s}\bigg|_{s=0}\nrjx(\tm^s) 
    = \dff\nrjx(\tm)[\zeta\circ\tm]
    = \intom \left\{\isom{L^2(\rfd)}{\dff\nrjx(\tm)}\right\}(\xi)\cdot[\zeta\circ\tm](\xi)\dd\rfd(\xi).
  \end{align*}
  And on the other hand, with $u:=\tm_\#\rfd$,
  \begin{align*}
    \frac{\dn}{\dn s}\bigg|_{s=0}\anrj(\tm^s_\#\rfd)
    = \dff\anrj(u)\left[\frac{\dn}{\dn s}\bigg|_{s=0}(\tm^s_\#\rfd)\right] 
   &  = \intom\left\{\isom{L^2(\dn x)}{\dff\anrj(u)}\right\}(x)\frac{\dn}{\dn s}\bigg|_{s=0}\big(\tm^s_\#\rfd\big)(x)\dd x \\ 
   & = \frac{\dn}{\dn s}\bigg|_{s=0}\intom\left\{\isom{L^2(\dn x)}{\dff\anrj(u)}\right\}\circ\tm^s(\xi)\dd\rfd(\xi) \\
   & = \intom\left\{\nabla\isom{L^2(\dn x)}{\dff\anrj(u)}\right\}\circ\tm(\xi)\cdot[\zeta\circ\tm](\xi)\dd\rfd(\xi).
  \end{align*}
  Recallin \eqref{eq:triveq}, and since $\zeta$ was arbitrary,
  the equality \eqref{eq:W2L2} follows.
\end{proof}

\begin{lem}
  \label{lem:el}
  The $L^2(\tm_\#\rfd)$-gradient of the functional $\nrjx$ defined in \eqref{eq:newpot}
  is given by
  \begin{align*}
    \velo(\T;\tm) 
    = \frac1{\tm_\#\rfd}\div\left[\prss\left(\frac{\tm_\#\rfd}{\det\dff\T}\right)\,(\dff\tm)^\#\right]
    +\nabla V\circ\T.
  \end{align*}
\end{lem}
\begin{proof}
  Fix $\T\in\nabla\xspc^+$.
  From the definition in \eqref{eq:newpot}, 
  we obtain for every ``variation'' $\xi\in\nabla\txspc$ and sufficiently small $\varepsilon$:
  \begin{align*}
    \nrjx\big((\T+\varepsilon\xi)\circ\tm\big)
    = \intom \left\{h_*\left(\frac{\det\dff(\T+\varepsilon\xi)}{\tm_\#\rfd}\right) + V(\T+\varepsilon\xi)\right\}\dd\tm_\#\rfd.   
  \end{align*}
  Now take the derivative with respect to $\varepsilon$ at $\varepsilon=0$.
  For simplification of the result,
  use that $h_*'(s)=-\prss(1/s)$, and perform an integration by parts:
  \begin{align*}
    \frac{\dn}{\dd\varepsilon}\bigg|_{\varepsilon=0}\nrjx\big((\T+\varepsilon\xi)\circ\tm\big)
    &= \intom \left\{h_*'\left(\frac{\det\dff\T}{\tm_\#\rfd}\right)\frac{\det\dff\T}{\tm_\#\rfd}\tr\big((\dff\T)^{-1}\dff\xi\big)
      + \xi\cdot\nabla V(\T)\right\}\dd\tm_\#\rfd \\
    &= -\intom\tr\left[\prss\left(\frac{\tm_\#\rfd}{\det\dff\T}\right)\,\big((\dff\T)^\#\big)^T\dff\xi\right]\dd y
      + \intom\xi\cdot\nabla V(\T)\dd\tm_\#\rfd \\
    &= \intom\left\{\frac1{\tm_\#\rfd}\div\left[\prss\left(\frac{\tm_\#\rfd}{\det\dff\T}\right)\,(\dff\T)^\#\right]
      + \nabla V\circ\T\right\}\cdot\xi\dd\tm_\#\rfd.
  \end{align*}
  This yields the result.
\end{proof}

\bibliography{dde2D}

\begin{thebibliography}{10}

\bibitem{agueh}
M.~Agueh and M.~Bowles.
\newblock One-dimensional numerical algorithms for gradient flows in the
  {$p$}-{W}asserstein spaces.
\newblock {\em Acta Appl. Math.}, 125:121--134, 2013.

\bibitem{AGS}
L.~Ambrosio, N.~Gigli, and G.~Savar{\'e}.
\newblock {\em Gradient flows in metric spaces and in the space of probability
  measures}.
\newblock Lectures in Mathematics ETH Z\"urich. Birkh\"auser Verlag, Basel,
  2005.

\bibitem{ALS}
L.~Ambrosio, S.~Lisini, and G.~Savar{\'e}.
\newblock Stability of flows associated to gradient vector fields and
  convergence of iterated transport maps.
\newblock {\em Manuscripta Math.}, 121(1):1--50, 2006.

\bibitem{carlier}
J.-D. Benamou, G.~Carlier, Q.~M{\'e}rigot, and E.~Oudet.
\newblock {D}iscretization of functionals involving the {M}onge-{A}mp\`ere
  operator (preprint).
\newblock {\em http://arxiv.org/abs/1408.4536}, 2014.

\bibitem{blanchet}
A.~Blanchet, V.~Calvez, and J.~A. Carrillo.
\newblock Convergence of the mass-transport steepest descent scheme for the
  subcritical {P}atlak-{K}eller-{S}egel model.
\newblock {\em SIAM J. Numer. Anal.}, 46(2):691--721, 2008.

\bibitem{budd}
C.~J. Budd, G.~J. Collins, W.~Z. Huang, and R.~D. Russell.
\newblock Self-similar numerical solutions of the porous-medium equation using
  moving mesh methods.
\newblock {\em R. Soc. Lond. Philos. Trans. Ser. A Math. Phys. Eng. Sci.},
  357(1754):1047--1077, 1999.

\bibitem{CarLis}
J.~A. Carrillo and S.~Lisini.
\newblock On the asymptotic behavior of the gradient flow of a polyconvex
  functional.
\newblock In {\em Nonlinear partial differential equations and hyperbolic wave
  phenomena}, volume 526 of {\em Contemp. Math.}, pages 37--51. Amer. Math.
  Soc., Providence, RI, 2010.

\bibitem{moll}
J.~A. Carrillo and J.~S. Moll.
\newblock Numerical simulation of diffusive and aggregation phenomena in
  nonlinear continuity equations by evolving diffeomorphisms.
\newblock {\em SIAM J. Sci. Comput.}, 31(6):4305--4329, 2009/10.

\bibitem{westdickenberg2}
F.~Cavalletti, M.~Sedjro, and M.~Westdickenberg.
\newblock A time discretization for the pressureless gas dynamics equations.
\newblock {\em arXiv preprint}, 2014.

\bibitem{naldi}
F.~Cavalli and G.~Naldi.
\newblock A {W}asserstein approach to the numerical solution of the
  one-dimensional {C}ahn-{H}illiard equation.
\newblock {\em Kinet. Relat. Models}, 3(1):123--142, 2010.

\bibitem{bertram}
B.~D{\"u}ring, D.~Matthes, and J.~P. Mili{\v{s}}i{\'c}.
\newblock A gradient flow scheme for nonlinear fourth order equations.
\newblock {\em Discrete Contin. Dyn. Syst. Ser. B}, 14(3):935--959, 2010.

\bibitem{Evans}
L.~C. Evans, O.~Savin, and W.~Gangbo.
\newblock Diffeomorphisms and nonlinear heat flows.
\newblock {\em SIAM J. Math. Anal.}, 37(3):737--751, 2005.

\bibitem{GosTos}
L.~Gosse and G.~Toscani.
\newblock Identification of asymptotic decay to self-similarity for
  one-dimensional filtration equations.
\newblock {\em SIAM J. Numer. Anal.}, 43(6):2590--2606 (electronic), 2006.

\bibitem{hadwiger}
H.~Hadwiger and D.~Ohmann.
\newblock Brunn-{M}inkowskischer {S}atz und {I}soperimetrie.
\newblock {\em Math. Z.}, 66:1--8, 1956.

\bibitem{JanSHesthaven:2007vy}
J.~S. Hesthaven, S.~Gottlieb, and D.~Gottlieb.
\newblock {Spectral methods for time-dependent problems, volume 21 of Cambridge
  Monographs on Applied and Computational Mathematics}, 2007.

\bibitem{maccamy}
R.~C. MacCamy and E.~Socolovsky.
\newblock A numerical procedure for the porous media equation.
\newblock {\em Comput. Math. Appl.}, 11(1-3):315--319, 1985.
\newblock Hyperbolic partial differential equations, II.

\bibitem{dde}
D.~Matthes and H.~Osberger.
\newblock Convergence of a variational {L}agrangian scheme for a nonlinear
  drift diffusion equation.
\newblock {\em ESAIM Math. Model. Numer. Anal.}, 48(3):697--726, 2014.

\bibitem{dlss}
D.~Matthes and H.~Osberger.
\newblock A convergent {L}agrangian discretization for a nonlinear fourth order
  equation.
\newblock {\em Accepted at Found. Comp. Math.}, 2015.

\bibitem{McCann}
R.~J. McCann.
\newblock A convexity principle for interacting gases.
\newblock {\em Adv. Math.}, 128(1):153--179, 1997.

\bibitem{OttPME}
F.~Otto.
\newblock The geometry of dissipative evolution equations: the porous medium
  equation.
\newblock {\em Comm. Partial Differential Equations}, 26(1-2):101--174, 2001.

\bibitem{roessler}
T.~Roessler.
\newblock Discretizing the porous medium equation based on its gradient flow
  structure --- a consistency paradox.
\newblock Technical report 150, Sonderforschungsbereich 611, May 2004.
\newblock Available online at
  \texttt{http://sfb611.iam.uni-bonn.de/uploads/150-komplett.pdf}.

\bibitem{russo}
G.~Russo.
\newblock Deterministic diffusion of particles.
\newblock {\em Comm. Pure Appl. Math.}, 43(6):697--733, 1990.

\bibitem{VazquezPME}
J.~L. V{\'a}zquez.
\newblock An introduction to the mathematical theory of the porous medium
  equation.
\newblock In {\em Shape optimization and free boundaries ({M}ontreal, {PQ},
  1990)}, volume 380 of {\em NATO Adv. Sci. Inst. Ser. C Math. Phys. Sci.},
  pages 347--389. Kluwer Acad. Publ., Dordrecht, 1992.

\bibitem{VilBook}
C.~Villani.
\newblock {\em Topics in optimal transportation}, volume~58 of {\em Graduate
  Studies in Mathematics}.
\newblock American Mathematical Society, Providence, RI, 2003.

\bibitem{westdickenberg}
M.~Westdickenberg and J.~Wilkening.
\newblock Variational particle schemes for the porous medium equation and for
  the system of isentropic {E}uler equations.
\newblock {\em M2AN Math. Model. Numer. Anal.}, 44(1):133--166, 2010.

\end{thebibliography}
\bibliographystyle{abbrv}

\end{document}